\documentclass[preprint,12pt,reqno]{elsarticle}
\usepackage[a4paper,margin=25mm]{geometry}
\usepackage[T1]{fontenc}
\usepackage{lmodern,amsmath,amssymb,amsthm,mathtools,microtype}
\usepackage{aliascnt}
\microtypesetup{expansion=false}
\input{glyphtounicode}
\usepackage[hidelinks]{hyperref}
\usepackage[nameinlink,capitalize]{cleveref}
\hypersetup{pdftitle={Generalized scalar curvature and Ricci flow on compact Lie groups},pdfauthor={Zhiqi Chen, Jianhua Yu, Huibin Chen, Junna Ni},pdfsubject={Confidential manuscript prepared for Advances in Mathematics}}
\numberwithin{equation}{section}
\newtheorem{maintheorem}{Theorem}

\newtheorem{theorem}{Theorem}[section]
\newaliascnt{proposition}{theorem}
\newtheorem{proposition}[proposition]{Proposition}
\aliascntresetthe{proposition}
\newaliascnt{lemma}{theorem}
\newtheorem{lemma}[lemma]{Lemma}
\aliascntresetthe{lemma}
\newaliascnt{corollary}{theorem}

\aliascntresetthe{corollary}
\theoremstyle{remark}
\newaliascnt{remark}{theorem}
\newtheorem{remark}[remark]{Remark}
\aliascntresetthe{remark}
\crefname{maintheorem}{Theorem}{Theorems}
\Crefname{maintheorem}{Theorem}{Theorems}
\crefname{proposition}{Proposition}{Propositions}
\crefname{lemma}{Lemma}{Lemmas}
\crefname{corollary}{Corollary}{Corollaries}
\crefname{remark}{Remark}{Remarks}
\DeclareMathOperator{\Ric}{Ric}
\DeclareMathOperator{\scal}{scal}
\DeclareMathOperator{\tr}{tr}
\DeclareMathOperator{\Ad}{Ad}
\DeclareMathOperator{\Sym}{Sym}
\DeclareMathOperator{\Isom}{Isom}
\DeclareMathOperator{\Vol}{Vol}
\newcommand{\id}{\mathrm{Id}}
\newcommand{\Sgen}{\mathcal S}
\newcommand{\Egen}{\mathcal E}
\newcommand{\gfr}{\mathfrak g}
\newcommand{\lfr}{\mathfrak l}
\newcommand{\Kill}{\mathfrak{isom}}
\newcommand{\HH}{\mathrm H}
\newcommand{\ext}{\Lambda}

\journal{Advances in Mathematics}

\begin{document}
\begin{frontmatter}
\title{Generalized scalar curvature and Ricci flow\\on compact Lie groups}

\author[gdut]{Zhiqi Chen}
\ead{chenzhiqi@gdut.edu.cn}
\author[scnu]{Jianhua Yu}
\ead{yujianhuascnu@126.com}
\author[njnu]{Huibin Chen}
\ead{chenhuibin@njnu.edu.cn}
\author[scnu]{Junna Ni\corref{cor1}}
\ead{nijunna@126.com}
\cortext[cor1]{Corresponding author.}

\affiliation[gdut]{organization={School of Mathematics and Statistics, Guangdong University of Technology},
                  city={Guangzhou},postcode={510520},country={P.R.~China}}
\affiliation[scnu]{organization={School of Mathematical Sciences, South China Normal University},
                  city={Guangzhou},postcode={510631},country={P.R.~China}}
\affiliation[njnu]{organization={Institute of Mathematics, School of Mathematical Sciences, Nanjing Normal University},
                  city={Nanjing},postcode={210023},country={P.R.~China}}

\begin{abstract}
Let $G$ be a compact connected semisimple Lie group, let $Q$ be a bi-invariant metric and let $\rho$ be its Cartan three-form. We prove that every left-invariant generalized Ricci flow with $[H_0]=\pm[\rho]$ exists for all positive time and converges smoothly to a Bismut--Ricci-flat pair without rescaling or pullback. The standard pair uniquely maximizes generalized scalar curvature in its torsion class, and every left-invariant Bismut--Ricci-flat pair on a compact simple group is standard, answering questions of Fusi, Lafuente and Stanfield and of Lauret and Will. On the product of two copies of a compact simple group, nonstandard stationary pairs in the Cartan class are known; we construct nonstationary flows with the exact initial form $H_0=\rho$ and nonstandard limits, disproving the general semisimple case of the convergence conjecture of Garcia-Fernandez and Streets. The proofs use a proper nonincreasing energy and a decomposition of Killing fields that parametrizes the stationary pairs by nilpotent Lie algebra endomorphisms.
\end{abstract}

\end{frontmatter}

\section{Introduction}
Generalized Ricci flow couples the evolution of a Riemannian metric $g$ to that of a closed three-form $H$. On a compact manifold it is given by
\begin{equation}\label{eq:grf}
 \dot g=-2F,\qquad \dot H=-d\delta_gH,\qquad
 F=\Ric(g)-\tfrac14H_g^2,
\end{equation}
Here $\delta_g$ is the formal adjoint of $d$. The symmetric tensor $H_g^2$ is defined in a $g$-orthonormal frame by $H_g^2(U,V)=\sum_{i,j}H(U,e_i,e_j)H(V,e_i,e_j)$. The form equation preserves the de Rham class of $H$. In this paper we study left-invariant solutions on compact semisimple Lie groups. We prove convergence in prescribed Cartan classes, determine the maximum of generalized scalar curvature and describe the stationary pairs. The distinction between simple groups and products then leads to different conclusions for the possible limits.

The natural geometric setting for the coupled equations is an exact Courant algebroid. Generalized metrics, described by a metric and a two-form, belong to the generalized geometry developed in \cite{Gu2011}. The Courant-algebroid formulation of Ricci flow and its compatibility with duality were studied by \v{S}evera and Valach~\cite{SeVa2017}, Garcia-Fernandez~\cite{Ga2019} and Streets~\cite{St2017}. Rubio and Tipler~\cite{RuTi2020} developed the automorphism group and the moduli space of generalized metrics. We use this language only to explain the fixed torsion class and its symmetries, and work throughout with the tensor pair $(g,H)$.

The analytic structure of generalized Ricci flow also extends several features of Ricci flow. Oliynyk, Suneeta and Woolgar~\cite{OlSuWo2006} studied its variational formulation, and Streets developed regularity and expanding entropy in~\cite{St2008} and weighted scalar-curvature and entropy identities in~\cite{St2023}. Pluriclosed flow provides a Hermitian source of related equations through the work of Streets and Tian~\cite{StTi2010,StTi2012}. Its homogeneous long-time behavior was studied by Arroyo and Lafuente~\cite{ArLa2019}, and non-K\"ahler Calabi--Yau geometry and pluriclosed flow were developed further by Garcia-Fernandez, Jordan and Streets~\cite{GaJoSt2023}. These relations motivate the study of stationary metrics with skew torsion, although we impose no complex structure here. We refer to the monograph~\cite{GaSt2021} for the general theory.

Homogeneity turns the geometric evolution into a finite-dimensional system, but does not remove questions of compactness or convergence. For ordinary Ricci flow, introduced by Hamilton~\cite{Ha1982}, the curvature formulas of Milnor~\cite{Mi1976} make the Lie-group case particularly explicit. Lauret characterized nilsolitons algebraically and used variational methods to construct Einstein solvmanifolds~\cite{La2001,La2002}. Bracket flow gives a complementary description of homogeneous Ricci flow~\cite{La2011,La2013}, and the algebraic structure of homogeneous solitons was studied by Lafuente and Lauret~\cite{LaLa2014}. B\"ohm and Lafuente obtained expanding soliton limits of parabolic blow-downs of immortal homogeneous Ricci flows~\cite{BoLa2018} and scalar-curvature-normalized convergence on solvmanifolds of real type~\cite{BoLa2019}. For generalized Ricci flow, the additional three-form evolves in a fixed cohomology class and must be controlled together with the metric.

Fusi, Lafuente and Stanfield~\cite{FuLaSt2024} developed homogeneous generalized Ricci flow and its bracket-flow formulation. They proved immortality of invariant flows on solvmanifolds and, for harmonic initial torsion, convergence of the rescaled generalized metrics on the universal cover of a nilmanifold to an expanding generalized Ricci soliton in the generalized Cheeger--Gromov topology. Their second paper~\cite{FuLaSt2026} proves subsequential convergence of rescaled nonflat invariant solutions on simply connected Lie groups with positive semidefinite Killing form and studies compact BRF pairs and their stability. Earlier work of Paradiso~\cite{Pa2021} also concerns generalized Ricci flow and algebraic solitons on nilpotent Lie groups. In this setting, Chen, Chen and Zhu~\cite{ChChZh2026a} proved that every left-invariant generalized nilsoliton has harmonic torsion in its preferred splitting. Their subsequent work~\cite{ChChZh2026b} obtains a Killing-form identity for algebraic generalized Ricci solitons, proves harmonicity when the Killing form is positive semidefinite without assuming unimodularity, and constructs examples with nonzero torsion on nilpotent Lie algebras admitting no classical nilsoliton. In the compact semisimple setting considered here, the Killing form is negative definite and the Cartan class supplies a different source of compactness.

The stationary pairs of \eqref{eq:grf} satisfy
\begin{equation}\label{eq:brf}
 \Ric(g)=\tfrac14H_g^2,\qquad dH=0,\qquad\delta_gH=0.
\end{equation}
They are called \emph{Bismut--Ricci-flat} (BRF), since these equations are equivalent to vanishing of the full Ricci tensors of $\nabla^\pm=\nabla^g\pm\tfrac12H^\sharp$, where $g(H^\sharp(U,V),W)=H(U,V,W)$. Flatness of $\nabla^\pm$ is a stronger condition. Compact homogeneous BRF pairs with nonzero curvature were constructed by Podest\`a and Raffero~\cite{PoRa2023,PoRa2024} and studied further by Lauret and Will~\cite{LaWi2023}, with the corrections recorded in their corrigendum. Guti\'errez~\cite{Gu2025} studied generalized Ricci flow in specified invariant families on aligned homogeneous spaces. Chen, Chen and Zhu~\cite{ChChZh2026c} formulated the invariant flow on compact homogeneous spaces in a fixed torsion class, including matrix variables for equivalent real-type isotropy summands. Their Ledger--Obata example shows that stability in a restricted diagonal family can coexist with a hyperbolic saddle in the full invariant system. These results show that the transitive action, torsion class and allowed invariant variations must be specified when comparing existence and stability statements.

Let $G$ be a nontrivial compact connected semisimple Lie group with Lie algebra $\gfr$, and let $Q$ be a positive bi-invariant metric. Its Cartan three-form and the associated constant are
\begin{equation}\label{eq:rho}
 \rho(X,Y,Z)=Q([X,Y],Z),\qquad
 c_Q=|\rho|_Q^2=\tr_Q(-B_{\gfr}),
\end{equation}
where $B_{\gfr}(X,Y)=\tr(\operatorname{ad}_X\operatorname{ad}_Y)$ is the Killing form. We use full tensor norms, so $|H|_g^2=\sum_{i,j,k}H(e_i,e_j,e_k)^2$. The form $\rho$ is closed and coclosed, and $\Ric(Q)=-\tfrac14B_{\gfr}=\tfrac14\rho_Q^2$. Thus $(Q,\rho)$ and $(Q,-\rho)$ are BRF. We call them the \emph{standard pairs}. The cohomological description of the Cartan classes comes from Lie algebra cohomology~\cite{ChEi1948} and is recalled in Section~\ref{sec:preliminaries}.

In what follows, all pairs are left-invariant and all cohomology groups have real coefficients. Fix $\varepsilon\in\{1,-1\}$. Our first result controls every trajectory with $[H_0]=\varepsilon[\rho]$.

\begin{maintheorem}\label{thm:flow}
For every left-invariant metric $g_0$ and closed left-invariant three-form $H_0$ with $[H_0]=\varepsilon[\rho]$, the solution of \eqref{eq:grf} from $(g_0,H_0)$ exists for all $t\geq0$. Its trajectory of pairs $(g_t,H_t)$ is contained in a compact subset of the space of positive metrics and closed three-forms in the prescribed class, and the entire solution converges smoothly to a BRF pair in that class. No rescaling or pullback by diffeomorphisms is needed.
\end{maintheorem}

The conclusion concerns the original tensor pairs on the fixed group. The limiting pair need not be standard when $G$ has several simple factors. The compactness assertion in Theorem~\ref{thm:flow} is also the starting point for the variational question considered next.

The generalized scalar curvature is
\begin{equation}\label{eq:scalar}
 \Sgen(g,H)=\scal(g)-\frac1{12}|H|_g^2.
\end{equation}
For invariant tensors it is a real-valued function on a finite-dimensional parameter space. Its critical points in a fixed torsion class on a compact Lie group are precisely the BRF pairs~\cite[Corollary~3.3]{FuLaSt2026}. The standard value is $\Sgen(Q,\pm\rho)=c_Q/6$. Fusi, Lafuente and Stanfield constructed non-bi-invariant BRF pairs on products in the same Cartan class and asked whether the standard value is nevertheless the global maximum~\cite[Theorem~E and Question~1.1]{FuLaSt2026}. The following theorem answers this question and determines the equality case.

\begin{maintheorem}\label{thm:maximum}
For every left-invariant metric $g$ and every closed left-invariant three-form $H$ with $[H]=\varepsilon[\rho]$, one has
\begin{equation}\label{eq:maximum}
 \Sgen(g,H)\leq\frac{c_Q}{6}.
\end{equation}
Equality holds if and only if $g=Q$ and $H=\varepsilon\rho$.
\end{maintheorem}

For $Q=-B_{\gfr}$ the maximum is $\dim G/6$. The statement requires neither volume normalization nor harmonicity of $H$. The metric $Q$ may have independent positive scales on the simple ideals. Thus the theorem applies to Cartan classes whose coefficients are nonzero and have the same sign.

To understand the other critical points, we establish a one-to-one correspondence between left-invariant BRF pairs in the fixed Cartan class $[H]=[\rho]$ and Lie algebra endomorphisms of $\gfr$ that are nilpotent as linear maps (Theorem~\ref{thm:parametrization}). The standard pair corresponds to the zero endomorphism. Since a simple Lie algebra has no nonzero nilpotent endomorphisms, this correspondence yields the rigidity result below. Lauret and Will asked whether every left-invariant BRF pair on a compact simple group is standard, after proving this when the metric is bi-invariant or the three-form is a fixed Cartan form~\cite[Introduction and Proposition~6.1]{LaWi2023}. We obtain the following answer.

\begin{maintheorem}\label{thm:rigidity}
Let $G$ be compact and connected with nonabelian simple Lie algebra, and set $Q_0=-B_{\gfr}$ and $\rho_0=Q_0([\cdot,\cdot],\cdot)$. Every left-invariant BRF pair on $G$ has the form
\begin{equation}\label{eq:simple-classification}
 g=aQ_0,\qquad H=\varepsilon a\rho_0,
 \qquad a>0,\quad\varepsilon\in\{1,-1\}.
\end{equation}
Conversely every such pair is BRF, and both connections $\nabla^\pm$ are flat. Bi-invariance holds for the original multiplication on $G$.
\end{maintheorem}

This conclusion includes every torsion class on a compact simple group. For a BRF pair, invariance of the metric already implies invariance of its harmonic three-form, because left translations act trivially on cohomology. Theorem~\ref{thm:rigidity} therefore also applies when left-invariance is initially assumed only for $g$.

The structure of the stationary set is closely related to stability, but does not determine every forward limit. Raffero and Vezzoni~\cite{RaVe2021} established dynamical stability and instability results near pairs consisting of a Ricci-flat metric and a vanishing three-form. Lee computed the second variation and proved linear stability for compact Bismut-flat pairs~\cite[Theorem~1.1 and Corollary~1.2]{Le2024}. He also proved linear stability of compact Bismut Hermitian--Einstein metrics in the pluriclosed setting~\cite{Le2025b} and established dynamical stability criteria under nonpositivity of the second variation and integrability of infinitesimal deformations~\cite{Le2025a}. In the invariant setting, Fusi, Lafuente and Stanfield proved local stability of standard pairs on compact simply connected semisimple groups and instability of certain nonstandard pairs on products~\cite[Theorems~D and E]{FuLaSt2026}. Beyond the compact homogeneous setting, Podest\`a and Raffero constructed complete positively curved steady generalized Ricci solitons on $\mathbb R^3$ with $\mathrm{SO}(3)$ symmetry~\cite{PoRa2025}. None of these local or self-similar results identifies the limit of every trajectory in a Cartan class.

Garcia-Fernandez and Streets conjectured that on a compact semisimple group every left-invariant solution with $H_0=\pm\rho$ exists for all positive time and converges to $(Q,\pm\rho)$~\cite[Conjecture~4.14, p.~74]{GaSt2021}. They verified this for $\mathrm{SU}(2)$ in~\cite[Example~4.15]{GaSt2021}. Theorems~\ref{thm:flow} and~\ref{thm:rigidity} prove the simple-group case under the weaker condition $[H_0]=\pm[\rho]$. On products, the nonstandard BRF pairs of~\cite{FuLaSt2026} show that fixing the class does not force the standard limit. The conjecture prescribes the initial form itself, so it remains to determine whether a nonstationary solution with $H_0=\rho$ can approach a nonstandard pair.

Let $K$ be compact and connected with nonabelian simple Lie algebra $\mathfrak k$, set $q=-B_{\mathfrak k}$ and take $Q=q\oplus q$ on $K\times K$. Define the left-invariant two-form
\begin{equation}\label{eq:product-beta}
 \beta((U,V),(U',V'))=q(V,U')-q(U,V').
\end{equation}
Matrices in the following statement act on the two factors of $\mathfrak k\oplus\mathfrak k$.

\begin{maintheorem}\label{thm:cartan}
There is $M_0>9$ such that, for every $M\ge M_0$, the solution with initial data $g_0=\left(\begin{smallmatrix}M&-3\\-3&1\end{smallmatrix}\right)\otimes q$ and $H_0=\rho,$ exists for all positive time and converges smoothly, without rescaling or pullback, to
\begin{equation}\label{eq:cartan-limit}
 g_*=\begin{pmatrix}3&-1\\-1&1\end{pmatrix}\otimes q,
 \qquad H_*=\rho+d\beta.
\end{equation}
\end{maintheorem}

The limit is the pair of~\cite[Proposition~7.3]{FuLaSt2026}, corresponding to $N(U,V)=(0,U)$ in Theorem~\ref{thm:parametrization}. If $d=\dim K$, its generalized scalar curvature is $d/4$, whereas the standard value is $d/3$. The two pairs therefore cannot be identified by a diffeomorphism. Theorem~\ref{thm:cartan} supplies nonstationary solutions with the precise initial condition in the conjecture and disproves its general semisimple case. Replacing $H$ by $-H$ gives examples in the negative class.

We conclude the introduction by describing the proofs. Generalized scalar curvature is nondecreasing along the flow, but this does not by itself control degeneration of the metric or growth of the three-form. For each right-invariant Killing field $X$, we construct the unique coclosed one-form $\xi_X$ satisfying $d\xi_X=\iota_XH$. In the class $[\rho]$, these primitives satisfy $\xi_X(Y)+\xi_Y(X)=2Q(X,Y)$. Summing $|X|_g^2+|\xi_X|_g^2$ over a $Q$-orthonormal basis of translation fields gives a proper function of $(g,H)$. Its dissipation identity keeps every trajectory in a compact subset of the parameter space. We then express the flow as an analytic gradient flow and apply the \L ojasiewicz inequality~\cite{Lo1963} to obtain convergence of the entire trajectory. This use of an analytic gradient inequality is related to the convergence method developed by Simon~\cite{Si1983}.

At a BRF pair, the same construction applies to all Killing fields. The map $X\mapsto\xi_X^\sharp$ is an involution whose eigenspaces consist of the fields parallel for $\nabla^+$ and $\nabla^-$. The positive pairing on translation fields gives a global parallel frame and an invariant inner product $R\geq Q$. This comparison bounds the critical values of $\Sgen$, and monotonicity gives Theorem~\ref{thm:maximum}. The two projections onto parallel fields also yield the identity $R(X,Y)-R(NX,NY)=Q(X,Y)$, which forces the endomorphism $N$ to be nilpotent and gives Theorem~\ref{thm:rigidity}. Flatness connects this description with metric connections with skew torsion~\cite{AgFr2010} and, in the Hermitian setting, with the Bismut-flat structures studied in~\cite{WaYaZh2020}. The realization by left and right translations is related to transitive actions on compact group manifolds~\cite{Oz1977} and to descriptions by complementary Lie subalgebras~\cite{BuDeVe2012}.

For Theorem~\ref{thm:cartan}, invariance under right translations by the diagonal subgroup reduces the flow on $K\times K$ to four variables. A large initial metric coefficient allows the three-form to move away from its Cartan value while the trajectory approaches an invariant surface. At a fixed cross-section, the transverse error is exponentially small in that coefficient, whereas the displacement on the attracting side has a polynomial lower bound. The comparison places the trajectory in a region converging to the nonstandard pair. The same reduced equations give the eternal solution from the nonstandard pair to the standard pair suggested in~\cite[Remark~7.8]{FuLaSt2026}.

Section~\ref{sec:preliminaries} fixes the geometric conventions and the parameter space. Section~\ref{sec:energy} constructs the proper energy and proves Theorem~\ref{thm:flow}. Section~\ref{sec:involution} describes the Killing fields of stationary pairs, proves Theorems~\ref{thm:maximum} and~\ref{thm:rigidity}, and gives the parametrization. Section~\ref{sec:cartan} proves Theorem~\ref{thm:cartan} and describes the eternal solution.
\section{Preliminaries}\label{sec:preliminaries}

We fix the conventions needed for the energy identities and the description of stationary pairs. We first explain how a generalized metric determines the pair $(g,H)$. We then record the norm and curvature formulas, and identify the finite-dimensional space on which the invariant flow evolves.

Fix a closed three-form $H_0$ on a manifold $M$. On $TM\oplus T^*M$, the natural pairing and the $H_0$-twisted Dorfman bracket are
\begin{equation}\label{eq:courant-data}
 \begin{aligned}
 \langle X+\xi,Y+\eta\rangle
       &=\tfrac12\bigl(\xi(Y)+\eta(X)\bigr),\\
 [X+\xi,Y+\eta]_{H_0}
       &=[X,Y]+\mathcal L_X\eta-\iota_Yd\xi+\iota_Y\iota_XH_0.
 \end{aligned}
\end{equation}
The data in~\eqref{eq:courant-data}, together with projection to $TM$, define the standard exact Courant algebroid associated with $H_0$. A generalized metric is a maximal positive subbundle for this pairing. Its projection to $TM$ is an isomorphism, since a nonzero pure covector has zero norm and cannot belong to a positive subbundle. Consequently it is the graph of a bundle map $A:TM\to T^*M$. The symmetric part of $A$ is a positive metric $g$, and the skew-symmetric part is a two-form $b$. Thus the generalized metric has the unique expression $V_+(g,b)=\{X+g(X,\cdot)+\iota_Xb:X\in TM\}$. The restriction of the pairing to this graph is $g$, because the two terms involving $b$ cancel. See \cite[Section~2.1]{FuLaSt2024} and \cite[Section~4]{RuTi2020}.

The two-form changes the splitting rather than the underlying geometric object. Define $e^{-b}(X+\xi)=X+\xi-\iota_Xb$. It sends $V_+(g,b)$ to $V_+(g,0)$ and satisfies
\begin{equation}\label{eq:splitting-change}
 [e^{-b}u,e^{-b}v]_{H_0+db}=e^{-b}[u,v]_{H_0}.
\end{equation}
For example, the terms involving $b$ on the left are obtained from $-\mathcal L_X\iota_Yb+\iota_Yd\iota_Xb=-\iota_{[X,Y]}b-\iota_Y\iota_Xdb$. The last term cancels the change in the twisting form, giving~\eqref{eq:splitting-change}. In the splitting with zero two-form component, the twisting form is therefore $H=H_0+db$. This is the passage to the tensor pair $(g,H)$ used below. See also \cite[Proposition~2.8 and Remark~2.9]{FuLaSt2024}.

The form $H$ is closed and has class $[H_0]$. Potentials that differ by a closed two-form determine the same pair $(g,H)$, so the uniqueness and convergence statements in this paper concern these pairs. In the $(g,b)$ description, the form equation is $\dot b=-\delta_g(H_0+db)$. Taking its exterior derivative gives $\dot H=-d\delta_gH$, which is the second equation in~\eqref{eq:grf}. This also explains directly why the torsion class is fixed along the flow.

We use $g$ to raise and lower indices, and write $\nabla=\nabla^g$ for the Levi-Civita connection. Tensor norms without an exterior marker include all ordered indices. In a $g$-orthonormal frame, the relation with the exterior norm is
\begin{equation}\label{eq:norm-conventions}
 |\omega|_g^2=\sum_{i_1,\ldots,i_p}\omega(e_{i_1},\ldots,e_{i_p})^2
            =p!\,|\omega|_{\ext,g}^2
 \qquad (\omega\in\Lambda^pT^*M).
\end{equation}
In particular, $|H|_g^2=6|H|_{\ext,g}^2$ and $|\omega|_g^2=2|\omega|_{\ext,g}^2$ for a two-form. The contraction $H_g^2$ defined in the introduction uses the full tensor convention, and $\tr_gH_g^2=|H|_g^2$. These factors account for the coefficients in both the BRF equation and generalized scalar curvature. The adjoint $\delta_g$ uses exterior inner products. We mark integrated exterior inner products by $\langle\cdot,\cdot\rangle_{L^2,\ext}$ and integrated full tensor inner products by $\langle\cdot,\cdot\rangle_{L^2,g}$.

For a one-form $\alpha$ and its metric dual $W=\alpha^\sharp$, the skew part of $\nabla\alpha$ is $\tfrac12d\alpha$. Its symmetric part is $\tfrac12\mathcal L_Wg$. Orthogonality of these two parts and~\eqref{eq:norm-conventions} give
\begin{equation}\label{eq:one-form-splitting}
 \begin{aligned}
 \nabla\alpha&=\operatorname{Sym}\nabla\alpha+\tfrac12d\alpha,
       &\mathcal L_Wg&=2\operatorname{Sym}\nabla\alpha,\\
 |\nabla\alpha|_g^2
       &=|\operatorname{Sym}\nabla\alpha|_g^2
                       +\tfrac12|d\alpha|_{\ext,g}^2.
 \end{aligned}
\end{equation}
We will use~\eqref{eq:one-form-splitting} when converting the Bochner formula into sums of squares. In particular, if $W$ is Killing, then $\delta_g\alpha=-\operatorname{div}_gW=0$ and $\nabla\alpha=\tfrac12d\alpha$.

Set $T=H^\sharp$, so that $g(T(U,V),W)=H(U,V,W)$. For $\sigma\in\{1,-1\}$, write $\nabla^\sigma=\nabla^g+\tfrac\sigma2T$. These connections preserve $g$, since $g(T(U,V),W)$ is skew-symmetric in $V,W$, and their torsions are $\sigma T$. We use the curvature convention $\mathcal R(U,V)W=\nabla_U\nabla_VW-\nabla_V\nabla_UW-\nabla_{[U,V]}W$ and $\Ric(U,V)=\sum_i g(\mathcal R(e_i,U)V,e_i)$. The difference-of-connections calculation gives
\begin{equation}\label{eq:connection-curvature}
 \begin{aligned}
 \mathcal R^\sigma(U,V)W
   ={}&\mathcal R^g(U,V)W
       +\frac\sigma2\bigl((\nabla_U T)(V,W)-(\nabla_V T)(U,W)\bigr)\\
     &+\frac14\bigl(T(U,T(V,W))-T(V,T(U,W))\bigr).
 \end{aligned}
\end{equation}
To take the Ricci contraction in~\eqref{eq:connection-curvature}, work in a normal orthonormal frame at the point under consideration. The derivative term is $\tfrac\sigma2\sum_i(\nabla_{e_i}H)(U,V,e_i)=-\tfrac\sigma2\delta_gH(U,V)$. One quadratic term vanishes because it contains $H(e_i,\cdot,e_i)$. The other is $-\tfrac14\sum_{i,j}H(U,e_i,e_j)H(V,e_i,e_j)$. Hence the full Ricci tensor is
\begin{equation}\label{eq:connection-ricci}
 \Ric^{\nabla^\pm}=\Ric(g)-\tfrac14H_g^2\mp\tfrac12\delta_gH.
\end{equation}
The first two terms in~\eqref{eq:connection-ricci} are symmetric and the last is skew-symmetric. Their separate vanishing is exactly the BRF system~\eqref{eq:brf}. See \cite[Section~2.1]{LaWi2023}. Moreover, at a BRF pair the trace gives $\scal(g)=\tfrac14|H|_g^2$, and therefore $\Sgen(g,H)=\tfrac16|H|_g^2$.

On a compact oriented manifold without boundary, stationarity of the form equation is equivalent to coclosedness of $H$. Indeed, $d\delta_gH=0$ implies $0=\langle d\delta_gH,H\rangle_{L^2,\ext}=\|\delta_gH\|_{L^2,\ext}^2$. A closed and coclosed form is the unique harmonic representative of its de Rham class. If an isometry is isotopic to the identity, it preserves the class and commutes with the Hodge Laplacian, so uniqueness forces it to preserve the harmonic representative. We will use this observation at BRF pairs to obtain primitives for every Killing field.

For a one-form $\alpha$, the integrated Weitzenb\"ock formula is
\begin{equation}\label{eq:weitzenbock}
 \int_M\bigl(|d\alpha|_{\ext,g}^2+|\delta_g\alpha|_g^2\bigr)\,dV_g
 =\int_M\bigl(|\nabla\alpha|_g^2+
               \Ric(\alpha^\sharp,\alpha^\sharp)\bigr)\,dV_g.
\end{equation}
Here $dV_g$ is the Riemannian volume form. The exterior norm on $d\alpha$ and the full norm on $\nabla\alpha$ in~\eqref{eq:weitzenbock} are different, as specified in~\eqref{eq:norm-conventions}. This distinction is needed when applying~\eqref{eq:one-form-splitting}.

We now specialize to the compact connected semisimple group $G$. A left-invariant tensor is determined by its value at the identity. Metrics are positive inner products on $\gfr$, and invariant $p$-forms are elements of $\Lambda^p\gfr^*$. On these forms, $d$ is the Chevalley--Eilenberg differential. For a two-form it is
\begin{equation}\label{eq:invariant-differential}
 d\beta(X,Y,Z)=-\beta([X,Y],Z)+\beta([X,Z],Y)-\beta([Y,Z],X).
\end{equation}
In~\eqref{eq:invariant-differential}, the vectors denote left-invariant fields and the bracket is that of $\gfr$. Thus $d$ is a fixed linear operator on the exterior algebra, independent of the evolving metric.

Let $d\mu$ be normalized Haar measure. Averaging a form by $\mathcal A\omega=\int_G L_a^*\omega\,d\mu(a)$ produces an invariant form and commutes with $d$. Each left translation is isotopic to the identity because $G$ is connected, so averaging acts as the identity on de Rham cohomology. Conversely, if an invariant form $\omega$ is exact, say $\omega=d\gamma$, then $\omega=d(\mathcal A\gamma)$ has an invariant primitive. These two observations identify the de Rham cohomology with the cohomology of invariant forms and justify working in a fixed affine space of invariant representatives.

For a semisimple Lie algebra, the first two Whitehead lemmas consequently give
\begin{equation}\label{eq:cohomology}
 \HH^1(G)=\HH^2(G)=0.
\end{equation}
See \cite[Sections~2, 12 and 21]{ChEi1948}. The two vanishings in~\eqref{eq:cohomology} have different roles below. The second provides a primitive of every closed contraction $\iota_XH$, while the first makes a coclosed primitive unique.

Write $\gfr=\bigoplus_{a=1}^s\gfr_a$ for the decomposition into simple ideals. For each ideal let $\rho_a=(-B_{\gfr_a})([\cdot,\cdot],\cdot)$, extended by zero on the other ideals. The classes $[\rho_a]$ form a basis of $\HH^3(G)$. To see the effect of the global group, pass to the finite simply connected cover, which is a product of compact simply connected simple groups. Each factor has one-dimensional third cohomology, and the degree-three K\"unneth formula has no mixed terms because the first and second cohomology groups vanish. Deck transformations are central translations, hence are isotopic to the identity. They act trivially on real cohomology, so pullback under the finite covering identifies the third cohomology groups. This gives the asserted basis on $G$ itself. See also \cite[Section~6]{LaWi2023}.

Every positive bi-invariant metric is of the form $Q=\bigoplus_a q_a(-B_{\gfr_a})$ with $q_a>0$. Indeed, different simple ideals are orthogonal for an invariant inner product, and its restriction to a compact simple ideal is a positive multiple of the negative Killing form. It follows that $[\rho]=\sum_a q_a[\rho_a]$. Thus the cohomology classes used in the paper retain a positive contribution from every simple ideal.

The geometry of the standard pair can also be read directly from the Lie algebra. For invariant fields, the Koszul formula is
\begin{equation}\label{eq:invariant-koszul}
 2g(\nabla^g_XY,Z)
   =g([X,Y],Z)-g([Y,Z],X)+g([Z,X],Y).
\end{equation}
Ad-invariance of $Q$ reduces~\eqref{eq:invariant-koszul} to $\nabla^Q_XY=\tfrac12[X,Y]$. The Cartan form is ad-invariant, so the three terms in its covariant derivative cancel and $\nabla^Q\rho=0$. It is therefore closed and coclosed. Since $\rho^\sharp(X,Y)=[X,Y]$, the connections of $(Q,\rho)$ satisfy $\nabla^+_XY=[X,Y]$ and $\nabla^-_XY=0$ in a left-invariant frame. The curvature of $\nabla^+$ is $[X,[Y,Z]]-[Y,[X,Z]]-[[X,Y],Z]=0$, and $\nabla^-$ is flat because its connection coefficients vanish. Left-invariant fields are $\nabla^-$-parallel. A right-invariant field has coefficient $\Ad_{x^{-1}}U$ in the left-invariant frame at $x$. Differentiating this coefficient in a left-invariant direction $X$ gives $-[X,\Ad_{x^{-1}}U]$, which cancels the connection term of $\nabla^+$. Hence right-invariant fields are $\nabla^+$-parallel.

We record explicitly how the adjoint is represented on invariant forms. Let $d_p$ be the matrix of $d:\Lambda^p\gfr^*\to\Lambda^{p+1}\gfr^*$ in fixed bases, and let $G_p(g)$ be the Gram matrix of the exterior inner product. Since normalized $g$-volume is $d\mu$, integration multiplies every invariant pointwise inner product by the same factor $\Vol(g)$. This factor cancels in the adjoint identity, giving
\begin{equation}\label{eq:invariant-adjoint}
 \delta_{p+1,g}=G_p(g)^{-1}d_p^{\mathsf T}G_{p+1}(g).
\end{equation}
The entries of $G_p(g)$ are polynomial expressions in the coefficients of $g^{-1}$. Formula~\eqref{eq:invariant-adjoint} therefore proves analytic dependence of the restricted adjoint on a positive metric. The restriction is legitimate because left translations are isometries and commute with $\delta_g$. Similarly, the Koszul formula~\eqref{eq:invariant-koszul} expresses the connection, and hence its Ricci tensor, by rational functions of the metric coefficients on the positive cone.

For fixed $Q$, put
\begin{equation}\label{eq:parameter-space}
 E=d(\Lambda^2\gfr^*),\qquad
 \mathcal M_\varepsilon
   =\operatorname{Sym}^2_+(\gfr^*)\times(\varepsilon\rho+E).
\end{equation}
Averaging shows that~\eqref{eq:parameter-space} contains exactly the invariant pairs in the prescribed cohomology class. Its tangent space at $(g,H)$ is $\operatorname{Sym}^2(\gfr^*)\times E$. The metric equation is a symmetric invariant two-tensor and the form equation is an exact invariant three-form. Thus the vector field of~\eqref{eq:grf} is tangent to $\mathcal M_\varepsilon$ and is a smooth finite-dimensional ODE there. The coefficient formulas above also show analyticity in these coordinates. A trajectory can be continued as long as it stays in a compact subset of the positive metric cone and the affine form space.

To relate coefficient convergence to the geometric statement, fix a smooth invariant frame on $G$. Every invariant tensor has constant components in that frame, and its components in coordinate frames are fixed smooth functions times those constants. Compactness of $G$ then shows that convergence of invariant tensor coefficients implies convergence in every $C^k$ norm. Hence coefficient convergence here is smooth convergence on the fixed manifold.

Changing $H$ to $-H$ leaves the BRF equations and $\Sgen$ unchanged, commutes with the flow and interchanges $\nabla^+$ with $\nabla^-$. It therefore suffices in the proofs to treat the positive class. We write $\mathcal M=\mathcal M_1$.

Finally, let $\lfr$ denote the space of right-invariant vector fields with the ordinary vector-field bracket. If $X^R$ is determined by $X\in\gfr$, then $[X^R,Y^R]=-[X,Y]^R$. Thus evaluation at the identity identifies $\lfr$ with the opposite Lie algebra $\gfr^{\mathrm{op}}$. The flow of $X^R$ is left translation by $\exp(tX)$, so these fields are Killing for every left-invariant metric. Although their individual lengths need not be constant for $g$, their $Q$-pairings are constant because $Q$ is bi-invariant. Moreover, $(L_a)_*X^R=(\Ad_aX)^R$, and the adjoint action is $Q$-orthogonal. These facts will make the sum defining the energy independent of the point of $G$.
\section{A proper energy and global convergence}\label{sec:energy}\label{sec:flow}

In this section we prove Theorem~\ref{thm:flow}. The first step is to control both the metric and the three-form by a single function on $\mathcal M$. We construct this function from the right-invariant Killing fields and their coclosed primitives. Its dissipation gives compactness of every trajectory, while the analytic gradient structure gives convergence to a single stationary pair. As in Section~\ref{sec:preliminaries}, we work in the positive class $[H]=[\rho]$.

For $X\in\lfr$, invariance and closedness of $H$ imply $d\iota_XH=\mathcal L_XH=0$. The vanishing of $\HH^2(G)$ therefore provides a primitive. Requiring this primitive to be coclosed removes its ambiguity, because $\HH^1(G)=0$. The equation $d\xi_X=\iota_XH$ is the infinitesimal isometry condition for a generalized metric in the splitting with zero two-form potential. See \cite[Eq.~(37)]{RuTi2020}. The coclosed lift is constructed using Hodge theory, as in \cite[proof of Proposition~4.8]{RuTi2020}. The class $[\rho]$ then determines the symmetric part of its pairing with the translation fields.

\begin{lemma}\label{lem:primitives}
For each $X\in\lfr$ there is a unique one-form $\xi_X$ such that
\begin{equation}\label{eq:primitive}
 d\xi_X=\iota_XH,\qquad \delta_g\xi_X=0.
\end{equation}
The assignment is linear, depends smoothly on $(g,H)$, and is equivariant under left translations. If $L_a(x)=ax$, then $L_a^*\xi_{(L_a)_*X}=\xi_X$. Moreover,
\begin{equation}\label{eq:positive-pairing}
 \xi_X(Y)+\xi_Y(X)=2Q(X,Y)\quad(X,Y\in\lfr).
\end{equation}
\end{lemma}
\begin{proof}
\emph{Construction and normalization.} Since $[H]=[\rho]$, averaging a primitive of $H-\rho$ over left translations gives a left-invariant two-form $\beta$ with $H=\rho+d\beta$. The map $d:\Lambda^2\gfr^*\to E$ is surjective. Fixing a linear right inverse allows us to choose $\beta$ linearly in $H-\rho$ throughout the parameter space.

For a right-invariant field $X$, put $\eta_X=Q(X,\cdot)$. The pairings of right-invariant fields under $Q$ are constant. If $Y,Z\in\lfr$, then $d\eta_X(Y,Z)=-Q(X,[Y,Z])$, where the bracket on the right is the vector-field bracket. This is $\rho(X,Y,Z)$ because evaluation identifies $\lfr$ with $\gfr^{\mathrm{op}}$. Thus $d\eta_X=\iota_X\rho$. The form $\beta$ is preserved by the flow of $X$, so Cartan's formula gives
\begin{equation}\label{eq:zeta}
 \zeta_X:=\eta_X-\iota_X\beta,
 \qquad d\zeta_X=\iota_XH.
\end{equation}
This choice has the required exterior derivative, but it need not be coclosed.

Let $d\mu=dV_g/\Vol(g)$ be normalized Haar measure. It is independent of the left-invariant metric $g$. We correct $\zeta_X$ by an exact one-form and impose zero mean on its potential. More precisely, solve
\begin{equation}\label{eq:primitive-hodge-gauge}
 \delta_gdf_X=-\delta_g\zeta_X,\qquad
 \int_G f_X\,d\mu=0,
 \qquad \xi_X=\zeta_X+df_X.
\end{equation}
The adjoint identity gives $\int_G\delta_g\zeta_X\,dV_g=\langle\zeta_X,d1\rangle_{L^2,g}=0$, so the right-hand side has zero mean. The kernel of $\delta_gd$ on functions consists of constants because $G$ is compact and connected. Hence \eqref{eq:primitive-hodge-gauge} has a unique smooth solution. If two one-forms satisfy \eqref{eq:primitive}, their difference is closed and coclosed. It is a harmonic one-form and vanishes by \eqref{eq:cohomology}. This proves uniqueness of $\xi_X$, independently of the auxiliary choice of $\beta$. Both equations in \eqref{eq:primitive-hodge-gauge} are linear in $X$, which proves linearity.

\emph{Equivariance and smooth dependence.} Fix $a\in G$. The field $(L_a)_*X$ is right-invariant, and naturality of contraction gives $L_a^*\iota_{(L_a)_*X}H=\iota_XH$. Since $L_a$ is an isometry of $g$, its pullback commutes with $\delta_g$. Thus $L_a^*\xi_{(L_a)_*X}$ satisfies \eqref{eq:primitive} for $X$ and equals $\xi_X$ by uniqueness. The same argument applied to \eqref{eq:primitive-hodge-gauge}, using invariance of $\beta$, $Q$ and $d\mu$, gives
\begin{equation}\label{eq:primitive-equivariance}
 L_a^*f_{(L_a)_*X}=f_X,
 \qquad L_a^*\xi_{(L_a)_*X}=\xi_X.
\end{equation}
For smooth dependence, the mean-zero functions form a fixed space because $d\mu$ is fixed. On this space $\delta_gd$ is invertible. Between mean-zero Sobolev spaces whose orders differ by two, its inverse depends smoothly on the coefficients of $g$, by elliptic estimates and the inverse-operator identity. The coefficients of $\zeta_X$ depend smoothly on $H$ and linearly on $X$. Equation~\eqref{eq:primitive-hodge-gauge}, followed by elliptic regularity, therefore gives smooth dependence of $\xi_X$ on $(g,H)$ in every differentiability order.

\emph{The pairing identity.} Let $F_t$ be the flow of $X\in\lfr$, so $F_t$ is a left translation. With the ordinary vector-field bracket, $\left.\frac{d}{dt}\right|_{t=0}(F_t)_*Y=-[X,Y]$. Differentiating the first identity in \eqref{eq:primitive-equivariance} therefore gives
\begin{equation}\label{eq:primitive-potential-pairing}
 X(f_Y)=f_{[X,Y]},\qquad
 X(f_Y)+Y(f_X)=0.
\end{equation}
On the other hand, skew-symmetry of $\beta$ gives $\zeta_X(Y)+\zeta_Y(X)=2Q(X,Y)$. The correction terms in \eqref{eq:primitive-hodge-gauge} have zero symmetric pairing by \eqref{eq:primitive-potential-pairing}. Adding them proves \eqref{eq:positive-pairing}.
\end{proof}

Fix a $Q$-orthonormal basis $X_1,\ldots,X_n$ of $\lfr$, where $n=\dim G$, and define
\begin{equation}\label{eq:energy}
 \Egen(g,H)=\sum_{i=1}^n\bigl(|X_i|_g^2+|\xi_{X_i}|_g^2\bigr).
\end{equation}
Linearity in Lemma~\ref{lem:primitives} shows that this sum is independent of the chosen orthonormal basis. Left translations preserve $g$ and the assignment $X\mapsto\xi_X$. Their action on $\lfr$ is $Q$-orthogonal, by bi-invariance of $Q$. Consequently the sum in \eqref{eq:energy} is constant on $G$ and defines a smooth real-valued function on $\mathcal M$. Individual summands need not be spatially constant.

The next calculation explains why the metric and the three-form cannot degenerate while $\Egen$ stays bounded. The positive pairing controls the inverse metric, and its remaining skew part determines the exact part of $H$.

\begin{proposition}\label{prop:properness}
Write $g_e=Q(P\cdot,\cdot)$ with $P$ positive definite and $Q$-self-adjoint. There is a $Q$-skew-adjoint linear map $K$ such that, with $A=\id-K$,
\begin{equation}\label{eq:primitive-matrix}
 (\xi_X)_e=Q(AX_e,\cdot),\qquad H=\rho+d\beta_K,
 \qquad (\beta_K)_e(U,V)=Q(KU,V).
\end{equation}
Here $\beta_K$ is extended left-invariantly, and the adjoints $A^*$ and $K^*$ below are taken with respect to $Q$. In these terms
\begin{equation}\label{eq:proper-energy}
 \Egen=\tr P+\tr\bigl(P^{-1}AA^*\bigr)
 =\tr P+\tr\bigl(P^{-1}(\id+K^*K)\bigr)
 \geq\tr P+\tr P^{-1}.
\end{equation}
Every sublevel of $\Egen$ is compact in $\mathcal M$, the parameter space of pairs defined in \eqref{eq:parameter-space}.
\end{proposition}
\begin{proof}
Define $A:\gfr\to\gfr$ by $Q(AX_e,\cdot)=(\xi_X)_e$. At the identity, \eqref{eq:positive-pairing} becomes $Q(AX_e,Y_e)+Q(AY_e,X_e)=2Q(X_e,Y_e)$. Thus $A+A^*=2\id$, so $A=\id-K$ with $K^*=-K$. The corresponding invariant two-form $\beta_K$ is uniquely determined by the last formula in \eqref{eq:primitive-matrix}.

The two families $\xi_X$ and $\eta_X-\iota_X\beta_K$ are equivariant under left translations and agree at the identity for every $X\in\lfr$. Given any $a\in G$, equivariance transports their values at $e$ to their values at $a$, since $(L_a)_*$ is an automorphism of the space of right-invariant fields. Hence the families agree everywhere. Differentiating this equality as in \eqref{eq:zeta} gives $\iota_XH=\iota_X(\rho+d\beta_K)$ for all $X\in\lfr$. These fields span every tangent space, so $H=\rho+d\beta_K$.

Put $e_i=(X_i)_e$. Then $\{e_i\}$ is a $Q$-orthonormal basis of $\gfr$. The squared dual norm of the covector $Q(V,\cdot)$ is $Q(P^{-1}V,V)$. Evaluating \eqref{eq:energy} at the identity therefore gives
\begin{equation}\label{eq:energy-trace-calculation}
 \sum_i|X_i|_g^2=\tr P,
 \qquad
 \sum_i|\xi_{X_i}|_g^2
   =\sum_iQ(P^{-1}Ae_i,Ae_i)
   =\tr(P^{-1}AA^*).
\end{equation}
Since $A=\id-K$ and $K^*=-K$, one has $AA^*=\id+KK^*=\id+K^*K$. Substitution in \eqref{eq:energy-trace-calculation} proves the equalities in \eqref{eq:proper-energy}. The remaining term has nonnegative trace because $P^{-1/2}K^*KP^{-1/2}$ is positive semidefinite. This proves the inequality.

Let $\Egen\leq C$, with $C>0$. If $\lambda$ is an eigenvalue of $P$, \eqref{eq:proper-energy} gives both $\lambda\leq C$ and $\lambda^{-1}\leq C$. Thus
\begin{equation}\label{eq:energy-coefficient-bounds}
 C^{-1}\id\leq P\leq C\id,
 \qquad
 |A|_Q^2=\tr(AA^*)\leq C^2,
 \qquad |K|_Q^2=|A|_Q^2-n\leq C^2-n.
\end{equation}
For the bound on $A$, use $P^{-1}\geq C^{-1}\id$ in the middle expression of \eqref{eq:proper-energy}. The identity for $|K|_Q^2$ follows from $AA^*=\id+K^*K$. If the last upper bound is negative, the sublevel is empty. Otherwise it bounds $K$ in the fixed finite-dimensional space of skew-adjoint endomorphisms. Since $K\mapsto d\beta_K$ is a fixed linear map, \eqref{eq:primitive-matrix} also bounds $H$ in fixed tensor coordinates.

It remains to check that these bounds give a closed set of pairs inside $\mathcal M$. Let $(P_j,H_j)$ be a sequence in the sublevel. After taking a subsequence, the coefficient bounds give limits $P_j\to P_\infty$ and $H_j\to H_\infty$. By \eqref{eq:energy-coefficient-bounds}, $P_\infty$ is positive definite. The affine space $\rho+E$ is closed, so $H_\infty\in\rho+E$. In particular, it is closed and has the prescribed cohomology class. Smooth dependence in Lemma~\ref{lem:primitives} gives $\Egen(P_\infty,H_\infty)\leq C$. Thus every sequence in the sublevel has a subsequence converging to a point of the same sublevel, which proves compactness.
\end{proof}

We next compute the derivative of $\Egen$ along the flow. For a vector field $X$, let $A_X(U,V)=g(\nabla^g_U X,V)$. When $X$ is Killing, $A_X$ is skew-symmetric and $dX^\flat=2A_X$. We write $\Sym$ for the symmetric part of a covariant two-tensor. All norms of covariant tensors in the dissipation formula are full tensor norms. The exterior norm is indicated explicitly whenever it is used in an adjoint identity.

\begin{proposition}\label{prop:dissipation}
Along an invariant solution of \eqref{eq:grf} with $[H_t]=[\rho]$, set $\xi_i=\xi_{X_i}$ and $W_i=\xi_i^\sharp$. Then
\begin{equation}\label{eq:dissipation}
 \frac{d\Egen}{dt}
 =-\frac{2}{\Vol(g_t)}\sum_{i=1}^n\int_G
 \left(\left|A_{X_i}-\tfrac12\iota_{W_i}H\right|_g^2
       +|\Sym\nabla\xi_i|_g^2\right)dV_g\leq0.
\end{equation}
\end{proposition}
\begin{proof}
\emph{Variation of the primitive and of its norm.} Fix one of the time-independent fields $X=X_i$. Write $\xi=\xi_X$, $W=\xi^\sharp$, and $\alpha=X^\flat$, all evaluated at the current time. Smooth dependence in Lemma~\ref{lem:primitives} allows us to differentiate $d\xi=\iota_XH$. The two-form $\delta_gH$ is left-invariant, because $g$ and $H$ are left-invariant and $\delta_g$ commutes with their isometries. Hence $\mathcal L_X\delta_gH=0$, and Cartan's formula gives $d\dot\xi=-\iota_Xd\delta_gH=d(\iota_X\delta_gH)$. The difference $\dot\xi-\iota_X\delta_gH$ is closed. Since $\HH^1(G)=0$, there is a function $u_X$ such that
\begin{equation}\label{eq:primitive-variation}
 \dot\xi=\iota_X\delta_gH+du_X.
\end{equation}
No equation for $u_X$ is needed. Its contribution to the integral vanishes, since $\langle du_X,\xi\rangle_{L^2,g}=\langle u_X,\delta_g\xi\rangle_{L^2,g}=0$.

For a metric variation $v=\dot g$, the derivative of the inverse metric is $-g^{-1}vg^{-1}$. Since $X$ is fixed while $\xi$ varies, this gives
\begin{equation}\label{eq:energy-norm-variation}
 \partial_t|X|_g^2=v(X,X),
 \qquad
 \partial_t|\xi|_g^2=-v(W,W)+2\langle\dot\xi,\xi\rangle_g.
\end{equation}
Insert $v=-2\Ric+\tfrac12H_g^2$ and \eqref{eq:primitive-variation} into \eqref{eq:energy-norm-variation}, and integrate at the fixed time under consideration. The result is
\begin{align}\label{eq:derivative-expanded}
 \int_G\partial_t(|X|_g^2+|\xi|_g^2)\,dV_g
 =\int_G\bigl(&-2\Ric(X,X)+\tfrac12H_g^2(X,X)
             +2\Ric(W,W)-\tfrac12H_g^2(W,W)\notag\\
             &+2\langle\iota_X\delta_gH,\xi\rangle_g\bigr)dV_g.
\end{align}
We have integrated a pointwise time derivative here. We have not differentiated an integral with a time-dependent volume form.

\emph{The curvature and mixed terms.} For the Killing field $X$, one has $\delta_g\alpha=0$ and $d\alpha=2A_X$. The norm convention in \eqref{eq:norm-conventions} gives $|d\alpha|_{\ext,g}^2=2|A_X|_g^2$, while $|\nabla\alpha|_g^2=|A_X|_g^2$. Applying \eqref{eq:weitzenbock} to $\alpha$ therefore yields
\begin{equation}\label{eq:killing-integrated}
 \int_G\Ric(X,X)\,dV_g=\int_G|A_X|_g^2\,dV_g.
\end{equation}
For $\xi$, one has $\delta_g\xi=0$ and $d\xi=\iota_XH$. In this case $|d\xi|_{\ext,g}^2=\tfrac12H_g^2(X,X)$, so a second application of \eqref{eq:weitzenbock} gives
\begin{equation}\label{eq:xi-weitzenbock}
 \int_G\Ric(W,W)\,dV_g
 =\int_G\left(\tfrac12H_g^2(X,X)-|\nabla\xi|_g^2\right)dV_g.
\end{equation}

For the mixed term in \eqref{eq:derivative-expanded}, exterior multiplication by $\alpha$ is adjoint to contraction by $X$. Integration by parts therefore gives
\begin{align}\label{eq:flow-cross}
 \int_G\langle\iota_X\delta_gH,\xi\rangle_g\,dV_g
 &=\langle\delta_gH,\alpha\wedge\xi\rangle_{L^2,\ext}\notag\\
 &=\langle H,d\alpha\wedge\xi-\alpha\wedge d\xi\rangle_{L^2,\ext}\notag\\
 &=\langle\iota_WH,d\alpha\rangle_{L^2,\ext}
       -\|\iota_XH\|_{L^2,\ext}^2\notag\\
 &=\int_G\left(\langle\iota_WH,A_X\rangle_g
                    -\tfrac12H_g^2(X,X)\right)dV_g.
\end{align}
In the third line, $d\alpha\wedge\xi=\xi\wedge d\alpha$ because $d\alpha$ has degree two. In the last line, $d\alpha=2A_X$ cancels the factor $1/2$ between the exterior and full inner products of two-forms. This accounts for the coefficient of the mixed product.

The symmetric and skew parts of $\nabla\xi$ are orthogonal. The skew part is $\tfrac12d\xi$, as recalled in \eqref{eq:one-form-splitting}. Hence
\begin{equation}\label{eq:sym-split}
 |\nabla\xi|_g^2=|\Sym\nabla\xi|_g^2+\tfrac14H_g^2(X,X).
\end{equation}
Substituting \eqref{eq:killing-integrated}--\eqref{eq:flow-cross} into \eqref{eq:derivative-expanded} gives
\begin{align}\label{eq:dissipation-before-square}
 \int_G\partial_t(|X|_g^2+|\xi|_g^2)\,dV_g
 =\int_G\bigl(&-2|A_X|_g^2-2|\nabla\xi|_g^2+\tfrac12H_g^2(X,X)\notag\\
 &-\tfrac12H_g^2(W,W)+2\langle\iota_WH,A_X\rangle_g\bigr)dV_g.
\end{align}
The $H_g^2(X,X)$ terms in \eqref{eq:dissipation-before-square} cancel by \eqref{eq:sym-split}. Since $|\iota_WH|_g^2=H_g^2(W,W)$, the remaining three terms involving $A_X$ and $W$ form the square
\begin{equation}\label{eq:dissipation-square}
 -2|A_X|_g^2+2\langle\iota_WH,A_X\rangle_g-\tfrac12|\iota_WH|_g^2
 =-2\left|A_X-\tfrac12\iota_WH\right|_g^2.
\end{equation}
Combining \eqref{eq:sym-split} and \eqref{eq:dissipation-square} gives
\begin{equation}\label{eq:single-field-dissipation}
 \int_G\partial_t(|X|_g^2+|\xi|_g^2)\,dV_g
 =-2\int_G\left(\left|A_X-\tfrac12\iota_WH\right|_g^2
                       +|\Sym\nabla\xi|_g^2\right)dV_g.
\end{equation}

\emph{Summation and the volume factor.} Sum \eqref{eq:single-field-dissipation} over $i$. By \eqref{eq:energy}, the pointwise sum on the left is $\partial_t\Egen(t)$, which is spatially constant. Its integral is therefore $\Vol(g_t)\dot\Egen(t)$. Division by $\Vol(g_t)$ proves \eqref{eq:dissipation}. Equivalently, differentiating $\int_G\Egen\,dV_g=\Egen\Vol(g)$ produces the same volume-variation term on both sides, and those terms cancel. No restriction on the evolution of the volume is used.
\end{proof}

Properness and dissipation now imply that every trajectory stays in a compact part of the parameter space. To identify its limits, we use the monotonicity of generalized scalar curvature. The formula in \cite[Eq.~(44)]{FuLaSt2024} takes the following form with our norm convention.

\begin{lemma}\label{lem:scalar-monotonicity}
Along an invariant solution,
\begin{equation}\label{eq:scalar-monotonicity}
 \frac{d\Sgen}{dt}=2|F|_g^2+|\delta_gH|_{\ext,g}^2\geq0.
\end{equation}
Equality holds precisely at a BRF pair.
\end{lemma}
\begin{proof}
We first compute the variation at an arbitrary invariant pair. Let $v$ be an invariant symmetric two-tensor and let $\eta$ be an invariant three-form variation. The scalar-curvature variation is $\dot{\scal}=-\langle\Ric,v\rangle_g$ plus divergence terms. Their sum is invariant and has zero integral on the compact group, so it vanishes pointwise. Thus $\dot{\scal}=-\langle\Ric,v\rangle_g$.

In a fixed frame, the full norm of $H$ is the contraction of two copies of $H$ against three copies of the inverse metric. Differentiating each copy of $H$ contributes the same term, and differentiating each inverse metric contributes the same contraction with $v$. Consequently
\begin{equation}\label{eq:torsion-norm-variation}
 \partial_t|H|_g^2=2\langle H,\eta\rangle_g-3\langle H_g^2,v\rangle_g.
\end{equation}
The factor $3$ in \eqref{eq:torsion-norm-variation} and the factor $6$ between the two norms of a three-form give
\begin{equation}\label{eq:scalar-first-variation}
 d\Sgen_{(g,H)}(v,\eta)
 =-\langle F,v\rangle_g-\tfrac16\langle H,\eta\rangle_g
 =-\langle F,v\rangle_g-\langle H,\eta\rangle_{\ext,g}.
\end{equation}
This formula will also identify the gradient of $\Sgen$ below.

Along the flow, substitute $v=-2F$ and $\eta=-d\delta_gH$ into \eqref{eq:scalar-first-variation}. All scalar contractions are invariant, so integration and division by $\Vol(g)$ preserve their pointwise values. The adjoint identity yields
\begin{equation}\label{eq:scalar-form-dissipation}
 \langle H,d\delta_gH\rangle_{\ext,g}
 =\frac{1}{\Vol(g)}\langle H,d\delta_gH\rangle_{L^2,\ext}
 =|\delta_gH|_{\ext,g}^2.
\end{equation}
Equations~\eqref{eq:scalar-first-variation} and \eqref{eq:scalar-form-dissipation} prove \eqref{eq:scalar-monotonicity}. The right-hand side vanishes exactly when $F=0$ and $\delta_gH=0$. Since $H$ is closed throughout the parameter space, these are precisely the BRF equations.
\end{proof}

\begin{proof}[Proof of Theorem~\ref{thm:flow}]
\emph{Global existence and stationary accumulation points.} The equations define a smooth ODE on $\mathcal M$. By Propositions~\ref{prop:properness} and~\ref{prop:dissipation}, the trajectory from $(g_0,H_0)$ stays in the compact sublevel $\{\Egen\leq\Egen(g_0,H_0)\}$. In particular, it stays a positive distance from the boundary of the metric cone, and its tensor coefficients remain bounded. The vector field of the ODE is smooth on a neighborhood of this compact set. The continuation theorem therefore excludes a finite maximal existence time.

By compactness, $\Sgen$ is bounded along the trajectory. Lemma~\ref{lem:scalar-monotonicity} implies that it increases to a finite limit $s_\infty$. Put $D=2|F|_g^2+|\delta_gH|_{\ext,g}^2$. Then $D\geq0$ and $\int_0^\infty D(t)\,dt=s_\infty-\Sgen(g_0,H_0)<\infty$. Both $D$ and the ODE vector field are smooth in the finite-dimensional parameters. Their restrictions to the compact sublevel show that $\dot D$ is bounded, so $D$ is uniformly continuous along the trajectory. It follows that $D(t)\to0$. Indeed, otherwise some $\epsilon>0$ and a sequence of disjoint intervals of fixed positive length would satisfy $D\geq\epsilon/2$ on each interval, contradicting integrability. Every accumulation point is consequently BRF and has scalar value $s_\infty$.

\emph{The analytic gradient structure.} Compactness alone does not imply that there is only one accumulation point. To prove this, we express the flow as a gradient flow on the parameter space. The metric on exact three-form variations must incorporate the operator $d\delta_g$ that occurs in the form equation.

Use $E=d(\Lambda^2\gfr^*)$ and $\mathcal M$ from \eqref{eq:parameter-space}. All tensor inner products in this part of the proof are evaluated at the identity. For invariant forms, their global $L^2$ inner products are multiplied by the common factor $\Vol(g)$. Hence the adjoint identities on invariant forms hold pointwise after canceling this factor, as in \eqref{eq:invariant-adjoint}. The operator $\delta_g$ preserves invariant forms. Thus $d\delta_g$ maps $E$ to itself, and we set $A_g=d\delta_g|_E$.

For $\eta,\zeta\in E$, the adjoint identity gives
\begin{equation}\label{eq:exact-form-operator}
 \langle A_g\eta,\zeta\rangle_{\ext,g}
 =\langle\delta_g\eta,\delta_g\zeta\rangle_{\ext,g},
 \qquad
 \langle A_g\eta,\eta\rangle_{\ext,g}=|\delta_g\eta|_{\ext,g}^2.
\end{equation}
Therefore $A_g$ is self-adjoint and nonnegative. If its quadratic form vanishes, choose an invariant two-form $\gamma$ with $\eta=d\gamma$. Then $|\eta|_{\ext,g}^2=\langle d\gamma,\eta\rangle_{\ext,g}=\langle\gamma,\delta_g\eta\rangle_{\ext,g}=0$. This proves that $A_g$ is positive definite and invertible on $E$.

Define an inner product on $T_{(g,H)}\mathcal M=\operatorname{Sym}^2(\gfr^*)\oplus E$ by
\begin{equation}\label{eq:parameter-metric}
 \mathbf h_{(g,H)}((v,\eta),(w,\zeta))
 =\tfrac12\langle v,w\rangle_g
       +\langle A_g^{-1}\eta,\zeta\rangle_{\ext,g}.
\end{equation}
Positivity of $A_g$ proves that \eqref{eq:parameter-metric} is a Riemannian metric. When $E=0$, the second term is absent. In fixed bases, the Gram matrices of the tensor inner products are analytic functions of the positive metric coefficients. Equation~\eqref{eq:invariant-adjoint} expresses $\delta_g$ using these Gram matrices and the fixed matrix of $d$, so its coefficients are analytic as well. Matrix inversion preserves analyticity on the set of invertible matrices. Thus $A_g^{-1}$ and $\mathbf h$ are analytic on all of $\mathcal M$.

We identify the gradient using \eqref{eq:scalar-first-variation}. Although $H$ need not belong to $E$, the form $d\delta_gH$ does. Given $\eta\in E$, write $\eta=A_g\psi$ with $\psi\in E$. Self-adjointness in \eqref{eq:exact-form-operator} and the adjoint identity give
\begin{align}\label{eq:form-gradient-pairing}
 \langle A_g^{-1}d\delta_gH,\eta\rangle_{\ext,g}
 &=\langle d\delta_gH,\psi\rangle_{\ext,g}\notag\\
 &=\langle\delta_gH,\delta_g\psi\rangle_{\ext,g}
 =\langle H,A_g\psi\rangle_{\ext,g}
 =\langle H,\eta\rangle_{\ext,g}.
\end{align}
Equations~\eqref{eq:parameter-metric} and \eqref{eq:form-gradient-pairing} show that the pairing of $(-2F,-d\delta_gH)$ with $(v,\eta)$ under $\mathbf h$ equals the right-hand side of \eqref{eq:scalar-first-variation}. Hence
\begin{equation}\label{eq:flow-analytic-gradient}
 \operatorname{grad}_{\mathbf h}\Sgen=(-2F,-d\delta_gH),
 \qquad
 |\operatorname{grad}_{\mathbf h}\Sgen|_{\mathbf h}^2
 =2|F|_g^2+|\delta_gH|_{\ext,g}^2.
\end{equation}
The second identity follows either from \eqref{eq:scalar-form-dissipation} or by evaluating the first variation on the gradient. The function $\Sgen$ is analytic, since curvature and the norm of $H$ are analytic in the invariant metric and form coefficients. Thus \eqref{eq:flow-analytic-gradient} realizes the given flow as the ascending gradient flow of an analytic function.

\emph{Convergence of the whole trajectory.} Write $z(t)=(g_t,H_t)$, and choose a BRF accumulation point $p$. We already know that $\Sgen(p)=s_\infty$. The \L ojasiewicz gradient inequality \cite{Lo1963}, in the local form also used in \cite[proof of Lemma~7.6]{FuLaSt2026}, gives a neighborhood of $p$, a constant $C>0$ and $0<\theta\leq\tfrac12$ such that
\begin{equation}\label{eq:lojasiewicz-gradient}
 |\Sgen(z)-\Sgen(p)|^{1-\theta}
 \leq C|\operatorname{grad}_{\mathbf h}\Sgen(z)|_{\mathbf h}.
\end{equation}
The usual coordinate form of this inequality is equivalent to \eqref{eq:lojasiewicz-gradient}, because $\mathbf h$ and a fixed Euclidean inner product are uniformly equivalent on a sufficiently small neighborhood. The inequality does not require $p$ to be an isolated critical point.

Set $q(t)=s_\infty-\Sgen(z(t))\geq0$. If $q$ vanishes at a finite time, monotonicity and the definition of $s_\infty$ imply that $\Sgen$ is constant thereafter. Equation~\eqref{eq:scalar-monotonicity} then makes the trajectory stationary. We may therefore suppose that $q(t)>0$. On every time interval during which $z(t)$ lies in the neighborhood of \eqref{eq:lojasiewicz-gradient}, equation~\eqref{eq:flow-analytic-gradient} gives
\begin{equation}\label{eq:lojasiewicz-length-derivative}
 -\frac{d}{dt}q(t)^\theta
 =\theta q(t)^{\theta-1}|\dot z(t)|_{\mathbf h}^2
 \geq\frac{\theta}{C}|\dot z(t)|_{\mathbf h}.
\end{equation}
Integrating \eqref{eq:lojasiewicz-length-derivative} from $t$ to $T$ yields
\begin{equation}\label{eq:lojasiewicz-length}
 \int_t^T|\dot z(s)|_{\mathbf h}\,ds
 \leq\frac C\theta\bigl(q(t)^\theta-q(T)^\theta\bigr)
 \leq\frac C\theta q(t)^\theta,
\end{equation}
provided the trajectory remains in that neighborhood on $[t,T]$.

We now verify this condition for all sufficiently large times. Choose a coordinate ball $B_{2r}(p)$ whose closure lies in the neighborhood of \eqref{eq:lojasiewicz-gradient} and in $\mathcal M$. There is $m>0$ such that $\mathbf h(v,v)\geq m|v|^2$ on this closure, where $|\cdot|$ is the Euclidean coordinate norm. A curve from $B_r(p)$ to the boundary of $B_{2r}(p)$ has $\mathbf h$-length at least $\sqrt m\,r$ while it remains in $\overline{B_{2r}(p)}$. Since $p$ is an accumulation point and $q(t)\to0$, choose $t_0$ so that $z(t_0)\in B_r(p)$ and $Cq(t_0)^\theta/\theta<\sqrt m\,r$. If a first exit time from $B_{2r}(p)$ existed, \eqref{eq:lojasiewicz-length} would give a strictly smaller length than this lower bound. Thus no exit occurs.

We may now let $T\to\infty$ in \eqref{eq:lojasiewicz-length}. The trajectory has finite remaining length and stays in the compact set $\overline{B_{2r}(p)}$. The lower comparison with the Euclidean norm makes it a Cauchy curve in these coordinates, so it converges. Since a subsequence converges to $p$, its limit is $p$. Finally, convergence of the coefficients of invariant tensors gives smooth convergence on the compact group. The argument used the original tensor pairs on the fixed group throughout. Replacing $H$ by $-H$ proves the assertion in the negative class, without rescaling or pullback by diffeomorphisms.
\end{proof}
\section{Stationary pairs, the sharp maximum and rigidity}\label{sec:involution}\label{sec:maximum}\label{sec:parametrization}
In this section we first describe the Killing fields of a BRF pair on a compact manifold. The argument uses harmonicity of $H$ and the vanishing of the first two cohomology groups. We then return to $G$, where the positive pairing on translation fields produces a global parallel frame. This gives the critical-value comparison and the algebraic data used to describe every stationary pair in the prescribed Cartan class.

Let $M$ be a connected compact oriented manifold without boundary satisfying $\HH^1(M)=\HH^2(M)=0$, and assume $(g,H)$ is BRF. Let $\Kill(M,g)=\{Z:\mathcal L_Zg=0\}$ be its Killing algebra, with the vector-field bracket, and let $\Isom_0(M,g)$ be the identity component of its isometry group. For $\sigma=\pm1$, write $\nabla^\sigma=\nabla^g+\frac\sigma2H^\sharp$. The next three statements concern this general manifold $M$. In their integrated identities, differential forms carry exterior norms, marked with $\ext$, whereas covariant derivatives and Lie derivatives carry full tensor norms.

\begin{lemma}\label{lem:killing-primitives}
For every $Z\in\Kill(M,g)$ there is a unique coclosed one-form $\xi_Z$ with $d\xi_Z=\iota_ZH$. This assignment is linear and $\Isom_0(M,g)$-equivariant.
\end{lemma}
\begin{proof}
Since $dH=\delta_gH=0$, the form $H$ is the harmonic representative of its cohomology class. Every isometry $a\in\Isom_0(M,g)$ acts trivially on cohomology, and $a^*H$ is again harmonic. Uniqueness of harmonic representatives therefore gives $a^*H=H$. Applying this to the flow of $Z$ yields $\mathcal L_ZH=0$. Cartan's formula now gives $d\iota_ZH=0$.

The vanishing of $\HH^2(M)$ provides a one-form $\theta$ with $d\theta=\iota_ZH$. To make it coclosed, solve $\delta_gdu=-\delta_g\theta$ with $\int_Mu\,dV_g=0$ and set $\xi_Z=\theta+du$. The equation is solvable because $\int_M\delta_g\theta\,dV_g=\langle\theta,d1\rangle_{L^2}=0$, and the kernel of $\delta_gd$ consists of constants. This is the same Hodge-theoretic construction as in \cref{lem:primitives}.

Two such primitives differ by a closed and coclosed one-form. Their difference vanishes because $\HH^1(M)=0$. Taking linear combinations proves linearity. If $a\in\Isom_0(M,g)$, naturality of contraction and of $\delta_g$ shows that $a^*\xi_{a_*Z}$ is a coclosed primitive of $\iota_ZH$. Uniqueness gives $a^*\xi_{a_*Z}=\xi_Z$, which is the required equivariance.
\end{proof}

\begin{proposition}\label{prop:involution}
The map $\Phi Z=\xi_Z^\sharp$ is an $\Isom_0(M,g)$-equivariant linear involution on the space of Killing fields, and $\Phi^2=\id$ and $\iota_{\Phi Z}H=dZ^\flat.$ More explicitly, for $W=\Phi Z$,
\begin{equation}\label{eq:killing-identity}
 \|\iota_WH-dZ^\flat\|_{L^2,\ext}^2
                +\tfrac12\|\mathcal L_Wg\|_{L^2,g}^2=0.
\end{equation}
\end{proposition}
\begin{proof}
We first apply the Bochner identity to an arbitrary coclosed one-form $\xi=W^\flat$. By \eqref{eq:one-form-splitting}, its covariant derivative satisfies $|\nabla\xi|_g^2=\frac12|d\xi|_{\ext,g}^2+\frac14|\mathcal L_Wg|_g^2$. Substituting this into \eqref{eq:weitzenbock} and moving the first term to the left gives
\begin{equation}\label{eq:BRF-bochner}
 \|d\xi\|_{L^2,\ext}^2-\tfrac12\|\mathcal L_Wg\|_{L^2,g}^2
 =2\int_M\Ric(W,W)\,dV_g=\|\iota_WH\|_{L^2,\ext}^2.
\end{equation}
For the last equality, the BRF equation gives $2\Ric(W,W)=\frac12H_g^2(W,W)=|\iota_WH|_{\ext,g}^2$. This accounts for the factor arising from the exterior norm of a two-form.

Now set $\alpha=Z^\flat$, $\xi=\xi_Z$ and $W=\xi_Z^\sharp$. Since $Z$ is Killing, $\delta_g\alpha=0$ and $\mathcal L_Zg=0$. Applying \eqref{eq:BRF-bochner} to $\alpha$, and then using $d\xi=\iota_ZH$, gives
\begin{equation}\label{eq:killing-equal-norms}
 \|d\alpha\|_{L^2,\ext}^2
 =\|\iota_ZH\|_{L^2,\ext}^2
 =\|d\xi\|_{L^2,\ext}^2.
\end{equation}
We also need the mixed term between $d\alpha$ and $\iota_WH$. The adjoint of exterior multiplication by a one-form is contraction by its metric dual. Since $d(\alpha\wedge\xi)=d\alpha\wedge\xi-\alpha\wedge d\xi$ and $\delta_gH=0$, integration by parts gives
\begin{equation}\label{eq:killing-cross-term}
 \begin{aligned}
 0&=\langle\delta_gH,\alpha\wedge\xi\rangle_{L^2,\ext}
   =\langle H,d(\alpha\wedge\xi)\rangle_{L^2,\ext}\\
  &=\langle\iota_WH,d\alpha\rangle_{L^2,\ext}
    -\langle\iota_ZH,d\xi\rangle_{L^2,\ext}.
 \end{aligned}
\end{equation}
Thus the mixed inner product in \eqref{eq:killing-cross-term} equals $\|d\xi\|_{L^2,\ext}^2$. Combining this with \eqref{eq:killing-equal-norms}, we obtain
\begin{equation}\label{eq:killing-square-cancellation}
 \begin{aligned}
 &\|\iota_WH-d\alpha\|_{L^2,\ext}^2
       +\tfrac12\|\mathcal L_Wg\|_{L^2,g}^2\\
 &\qquad=\|\iota_WH\|_{L^2,\ext}^2
          -\|d\xi\|_{L^2,\ext}^2
          +\tfrac12\|\mathcal L_Wg\|_{L^2,g}^2=0.
 \end{aligned}
\end{equation}
The last equality is \eqref{eq:BRF-bochner} applied to $\xi$. Hence \eqref{eq:killing-square-cancellation} proves \eqref{eq:killing-identity}.

Both terms in \eqref{eq:killing-identity} are nonnegative. Their vanishing implies pointwise that $W$ is Killing and $\iota_WH=d\alpha$. Moreover, $\alpha$ is coclosed. By \cref{lem:killing-primitives}, it is the canonical primitive associated with $W$, so $\xi_W=\alpha$. Taking metric duals gives $\Phi W=Z$. This proves both that $\Phi$ maps Killing fields to Killing fields and that $\Phi^2=\id$. Linearity and equivariance follow from \cref{lem:killing-primitives} and metric duality.
\end{proof}

\begin{proposition}\label{prop:parallel-fields}
The eigenspaces $\mathfrak p_\sigma=\ker(\Phi-\sigma\id)$ are commuting ideals, and
\begin{equation}\label{eq:killing-decomposition}
 \Kill(M,g)=\mathfrak p_+\oplus\mathfrak p_-,\qquad
 \mathfrak p_\sigma=\{V:\nabla^\sigma V=0\},\qquad
 \dim\mathfrak p_\sigma\leq\dim M.
\end{equation}
The projections $\pi_\pm=\tfrac12(\id\pm\Phi)$ are Lie algebra homomorphisms. Evaluation at any point is injective on each $\mathfrak p_\sigma$.
\end{proposition}
\begin{proof}
Differentiating equivariance along the flow of a Killing field $Z$ gives $\Phi[Z,V]=[Z,\Phi V]$. If $\Phi V=\sigma V$, then $\Phi[Z,V]=\sigma[Z,V]$, so each eigenspace is an ideal. Since $\Phi^2=\id$, these eigenspaces give a direct-sum decomposition. The bracket $[\mathfrak p_+,\mathfrak p_-]$ lies in both ideals and therefore vanishes. Consequently, for Killing fields $A,B$, one has $\pi_\sigma[A,B]=[\pi_\sigma A,\pi_\sigma B]$.

It remains to identify the eigenspaces with parallel fields, keeping track of the sign of the torsion. If $\Phi V=\sigma V$, then $\xi_V=\sigma V^\flat$ and $dV^\flat=\sigma\iota_VH$. Since $V$ is Killing, its covariant derivative is skew-symmetric. For arbitrary tangent fields $U,W$, we therefore have
\begin{equation}\label{eq:eigenfield-parallel-sign}
 \begin{aligned}
 2g(\nabla^g_UV,W)
 &=dV^\flat(U,W)=\sigma H(V,U,W)\\
 &=-\sigma H(U,V,W),
 \qquad g(\nabla^\sigma_UV,W)=0.
 \end{aligned}
\end{equation}
Conversely, if $\nabla^\sigma V=0$, then $2g(\nabla^g_UV,W)=-\sigma H(U,V,W)$ directly from the definition of the connection. The right-hand side is skew-symmetric in $U,W$, so $V$ is Killing. Its dual is coclosed, and the exterior-derivative relation in \eqref{eq:eigenfield-parallel-sign} gives $d(\sigma V^\flat)=\iota_VH$. Uniqueness in \cref{lem:killing-primitives} implies $\xi_V=\sigma V^\flat$, so $V\in\mathfrak p_\sigma$.

For any connection, a parallel field is determined by its value at one point on a connected manifold, by parallel transport along curves. Hence evaluation is injective on $\mathfrak p_\sigma$, and its dimension is at most $\dim M$. This completes \eqref{eq:killing-decomposition}.
\end{proof}

We now return to the nontrivial compact connected semisimple group $G$ and the positive class $[\rho]$. The cohomology vanishings in \eqref{eq:cohomology} allow us to apply the preceding statements. We also return to the full tensor convention for $|H|_g^2$. The right-invariant translation fields form the $n$-dimensional Killing subalgebra $\lfr$, where $n=\dim G$. Evaluation at the identity identifies its bracket with the opposite of the bracket on $\gfr$. On a general manifold the parallel spaces in \eqref{eq:killing-decomposition} may have dimension less than $n$. Here the positive pairing from \eqref{eq:positive-pairing} forces $\mathfrak p_+$ to have dimension $n$ and makes it possible to compare its constant metric with $Q$.

\begin{proposition}\label{prop:critical-bound}
If a left-invariant pair $(g,H)$ in $[\rho]$ is BRF, then there is a positive $\Ad$-invariant inner product $R$ on $\lfr\simeq\gfr^{\mathrm{op}}$ with $R\geq Q$ and
\begin{equation}\label{eq:critical-value}
 \Sgen(g,H)=\frac16\tr_R(-B_{\lfr})
                  \leq\frac16\tr_Q(-B_{\lfr})=\frac{c_Q}{6}.
\end{equation}
Equality holds precisely at $(Q,\rho)$. Both connections $\nabla^\pm$ of every such BRF pair are flat.
\end{proposition}
\begin{proof}
\emph{The parallel frame and its metric.} By uniqueness, the coclosed primitives for $X\in\lfr$ agree with those constructed for all Killing fields in \cref{lem:killing-primitives}. In particular, their positive pairing remains valid. Using $\Phi=\pi_+-\pi_-$ and $X=\pi_+X+\pi_-X$, we can rewrite \eqref{eq:positive-pairing} as
\begin{equation}\label{eq:projection-pairing}
 Q(X,Y)=g(\pi_+X,\pi_+Y)-g(\pi_-X,\pi_-Y).
\end{equation}
Indeed, $2Q(X,Y)=g(\Phi X,Y)+g(X,\Phi Y)$, and the two mixed terms cancel in this sum. No orthogonality between $\mathfrak p_+$ and $\mathfrak p_-$ is needed. The two pairings on the right of \eqref{eq:projection-pairing} are constant because each connection $\nabla^\pm$ is metric and its corresponding fields are parallel.

If $\pi_+X=0$, equation~\eqref{eq:projection-pairing} gives $Q(X,X)=-|\pi_-X|_g^2\leq0$. Since $Q$ is positive definite, $X=0$. Thus $\pi_+|_{\lfr}$ is injective. Combining $\dim\lfr=n$ with \cref{prop:parallel-fields}, we obtain $\dim\mathfrak p_+=n$, and $\pi_+|_{\lfr}$ is a Lie algebra isomorphism. Evaluation on $\mathfrak p_+$ is injective at every point, hence is an isomorphism at every point. A basis of $\mathfrak p_+$ therefore gives a global $\nabla^+$-parallel frame.

Define $R(X,Y)=g(\pi_+X,\pi_+Y)$ by transporting this constant metric to $\lfr$. This inner product is positive definite, and \eqref{eq:projection-pairing} gives $R-Q=g(\pi_-\cdot,\pi_-\cdot)\geq0$. The definition uses the projection before evaluation, so $R$ need not equal the original inner product $g_e$ on translation fields.

For $U,V,W\in\mathfrak p_+$, the function $g(V,W)$ is constant and $U$ is Killing. Evaluating the vanishing Lie derivative of the metric gives
\begin{equation}\label{eq:parallel-metric-invariance}
 g([U,V],W)+g(V,[U,W])=0.
\end{equation}
Transporting \eqref{eq:parallel-metric-invariance} by the Lie algebra isomorphism $\pi_+|_{\lfr}$ proves that $R$ is ad-invariant. Since $G$ is connected, it is also $\Ad$-invariant under the identification with $\gfr^{\mathrm{op}}$.

\emph{The critical value and equality case.} The torsion of $\nabla^+$ is $H^\sharp$. Since $U,V\in\mathfrak p_+$ are parallel, its definition gives
\begin{equation}\label{eq:parallel-bracket}
 H^\sharp(U,V)=-[U,V].
\end{equation}
Choose an $R$-orthonormal basis $u_1,\ldots,u_n$ of $\lfr$ and put $U_i=\pi_+u_i$. The fields $U_i$ form a global $g$-orthonormal frame. Equation~\eqref{eq:parallel-bracket} and ad-invariance of $R$ give
\begin{equation}\label{eq:parallel-torsion-norm}
 \begin{aligned}
 |H|_g^2
 &=\sum_{i,j,k}g([U_i,U_j],U_k)^2
   =\sum_{i,j}|[u_i,u_j]|_R^2\\
 &=\sum_i\tr(\operatorname{ad}_{u_i}^*\operatorname{ad}_{u_i})
   =-\sum_i\tr(\operatorname{ad}_{u_i}^2)
   =\tr_R(-B_{\lfr}).
 \end{aligned}
\end{equation}
The adjoint in this formula is taken with respect to $R$, and $\operatorname{ad}_{u_i}^*=-\operatorname{ad}_{u_i}$. Taking the trace of the BRF equation gives $\scal(g)=|H|_g^2/4$. Hence \eqref{eq:parallel-torsion-norm} proves the first equality in \eqref{eq:critical-value}. Reversing a Lie bracket changes every adjoint operator by a minus sign, so it leaves the Killing form unchanged. In particular, $B_{\lfr}=B_{\gfr}$ under evaluation at the identity.

To compare the two traces, let $A=Q^{-1}R\geq\id$ and $C=Q^{-1}(-B_{\lfr})>0$. Thus $\tr_R(-B_{\lfr})=\tr(A^{-1}C)$. We obtain
\begin{equation}\label{eq:critical-trace-comparison}
 \begin{aligned}
 \tr_Q(-B_{\lfr})-\tr_R(-B_{\lfr})
 &=\tr\bigl((\id-A^{-1})C\bigr)\\
 &=\tr\bigl(C^{1/2}(\id-A^{-1})C^{1/2}\bigr)\geq0.
 \end{aligned}
\end{equation}
The operator inside the last trace is nonnegative. Since $C$ is positive definite, equality in \eqref{eq:critical-trace-comparison} forces $\id-A^{-1}=0$, and therefore $R=Q$. This proves the inequality in \eqref{eq:critical-value} and identifies when equality can occur.

If $R=Q$, equation~\eqref{eq:projection-pairing} gives $\pi_-X=0$ for every $X\in\lfr$. Hence $X=\pi_+X$ and $g(X,Y)=Q(X,Y)$ on translation fields. These fields span every tangent space, so the original metric is $g=Q$. Also $\Phi X=X$, and therefore $\xi_X=X^\flat=\eta_X$. The identity $\iota_XH=d\eta_X=\iota_X\rho$ now implies $H=\rho$. Conversely, the standard pair has $R=Q$ and attains equality.

\emph{Flatness of both connections.} The global parallel frame makes $\nabla^+$ flat. For $U,V\in\mathfrak p_+$, equation~\eqref{eq:parallel-bracket} also gives $\nabla^-_UV=\nabla^+_UV-H^\sharp(U,V)=[U,V]$. Since $\mathfrak p_+$ is a Lie algebra of vector fields, its brackets have constant coefficients in a fixed basis. Thus for $U,V,W\in\mathfrak p_+$ the curvature of $\nabla^-$ is
\begin{equation}\label{eq:opposite-connection-flatness}
 \begin{aligned}
 \mathcal R^-(U,V)W
 &=\nabla^-_U\nabla^-_VW-\nabla^-_V\nabla^-_UW-\nabla^-_{[U,V]}W\\
 &=[U,[V,W]]-[V,[U,W]]-[[U,V],W]=0.
 \end{aligned}
\end{equation}
The last equality is the Jacobi identity. Since the fields in $\mathfrak p_+$ span each tangent space, \eqref{eq:opposite-connection-flatness} proves that $\nabla^-$ is flat on $G$. This argument does not require a global parallel frame for $\nabla^-$ on a possibly non-simply-connected group.
\end{proof}

\begin{proof}[Proof of \cref{thm:maximum}]
Replacing $H$ by $-H$ leaves the metric equation, the BRF condition and $\Sgen$ unchanged, and commutes with the form equation. It is therefore enough to consider $[H]=[\rho]$.

Start \eqref{eq:grf} at $(g,H)$. The compactness argument in Section~\ref{sec:flow} gives times $t_j\to\infty$ for which $(g_{t_j},H_{t_j})$ converges to a BRF pair $(g_\infty,H_\infty)$ in the same class. Monotonicity and \cref{prop:critical-bound} imply, for every fixed $t\geq0$,
\begin{equation}\label{eq:maximum-from-critical-bound}
 \begin{aligned}
 \Sgen(g,H)&\leq\Sgen(g_t,H_t)
 \leq\lim_{j\to\infty}\Sgen(g_{t_j},H_{t_j})\\
 &=\Sgen(g_\infty,H_\infty)\leq\frac{c_Q}{6}.
 \end{aligned}
\end{equation}
Here the second inequality follows by taking $j$ large enough that $t_j\geq t$. This proves the maximum bound for every pair in the class, including pairs that are not stationary.

If the initial value equals $c_Q/6$, all inequalities in \eqref{eq:maximum-from-critical-bound} are equalities, so $\Sgen(g_t,H_t)$ is constant. At $t=0$, equation~\eqref{eq:scalar-monotonicity} gives $F=0$ and $\delta_gH=0$. Since $H$ is closed, the initial pair is BRF. The equality case in \cref{prop:critical-bound} gives $(g,H)=(Q,\rho)$. The standard pair attains the bound, completing the proof in both sign classes.
\end{proof}

\begin{remark}\label{rem:scaling}
Under the hypotheses of \cref{thm:maximum}, if $r=\scal(g)>0$, then
\begin{equation}\label{eq:scalar-torsion-bound}
 |H|_g^2\geq\frac{64}{c_Q^2}r^3,
\end{equation}
with equality exactly when $g=aQ$ for some $a>0$ and $H=\varepsilon\rho$. To see this, put $h=|H|_g>0$ and vary the metric scale while keeping $H$ fixed, so its cohomology class is unchanged. The maximum over this one-dimensional family is attained at $s_*=h/(2\sqrt r)$ and satisfies
\begin{equation}\label{eq:optimal-scalar-scale}
 \max_{s>0}\Sgen(sg,H)
 =\max_{s>0}\left(\frac r s-\frac{h^2}{12s^3}\right)
 =\frac{4r^{3/2}}{3h}\leq\frac{c_Q}{6}.
\end{equation}
The last inequality is \cref{thm:maximum}. Rearranging \eqref{eq:optimal-scalar-scale} gives \eqref{eq:scalar-torsion-bound}. Equality holds precisely when $(s_*g,H)=(Q,\varepsilon\rho)$, which gives the stated equality case. Conversely, every such scaled standard pair has equality.
\end{remark}

We now use the second projection to determine all stationary pairs. The projection onto the positive eigenspace determines the inner product $R$ used in the bound. The other projection records how the original translations differ from the standard action. Comparing their values at the identity produces a Lie algebra endomorphism $N$ with $R-N^*R=Q$, where $(N^*R)(X,Y)=R(NX,NY)$. The identity makes this endomorphism strictly contractive, so the following lemma applies.

\begin{lemma}\label{lem:nilpotence}
If a Lie algebra endomorphism $N$ of a compact semisimple Lie algebra is strictly contractive for some positive inner product, then $N$ is nilpotent.
\end{lemma}
\begin{proof}
Let $\langle\cdot,\cdot\rangle$ be an inner product for which $N$ is strictly contractive. Compactness of its unit sphere gives a constant $c<1$ such that $|NX|\leq c|X|$ for every $X$.

The subspaces $N^k\gfr$ form a descending sequence. Their dimensions eventually stabilize, so for some $k$ we have $N^k\gfr=N^{k+1}\gfr=\mathfrak h$. Each $N^j$ is a Lie algebra homomorphism, and its kernel is an ideal. Hence $N^j\gfr\simeq\gfr/\ker N^j$ is compact semisimple. In particular, $\mathfrak h$ is compact semisimple and the restriction $A=N|_{\mathfrak h}$ is a surjective endomorphism, hence an automorphism.

Suppose $\mathfrak h\ne0$. The relation $\operatorname{ad}_{AX}=A\operatorname{ad}_XA^{-1}$ implies $B_{\mathfrak h}(AX,AY)=B_{\mathfrak h}(X,Y)$. Thus $A$ preserves the positive inner product $-B_{\mathfrak h}$, and $|\det A|=1$. On the other hand, every singular value of $A$ for the restriction of $\langle\cdot,\cdot\rangle$ is at most $c$. Therefore $|\det A|\leq c^{\dim\mathfrak h}<1$, a contradiction. We conclude that $\mathfrak h=0$, so $N^k=0$.
\end{proof}

\begin{theorem}\label{thm:parametrization}
For the fixed pair $(G,Q)$, the left-invariant BRF pairs with $[H]=[\rho]$ are in bijection with the Lie algebra endomorphisms $N:\gfr\to\gfr$ that are nilpotent as linear maps. If $N^m=0$, set
\begin{equation}\label{eq:parametrization-data}
 R_N(X,Y)=\sum_{j=0}^{m-1}Q(N^jX,N^jY),\qquad L=\id-N.
\end{equation}
The corresponding tensors, specified at the identity and extended by left translations, are
\begin{equation}\label{eq:brf-pairs}
 \begin{aligned}
 g_N(X,Y)&=R_N(LX,LY),\\
 H_N(X,Y,Z)&=R_N([LX,LY],LZ).
 \end{aligned}
\end{equation}
Both connections $\nabla^\pm$ are flat, and $\Sgen(g_N,H_N)=\tfrac16\tr_{R_N}(-B_{\gfr})$. The negative class is obtained by replacing $H_N$ with $-H_N$.
\end{theorem}

This parametrization is for tensor pairs with the given multiplication on $G$, without taking a quotient by automorphisms or diffeomorphisms. The standard pair corresponds to $N=0$.

\begin{proof}[Proof of \cref{thm:parametrization}]
\emph{Necessity and the sign of the endomorphism.} Let $T=\pi_+|_{\lfr}\colon\lfr\to\mathfrak p_+$ be the Lie algebra isomorphism from \cref{prop:critical-bound}. Evaluation at $e$ is an isomorphism from $\mathfrak p_+$ onto $T_eG$. Consequently, for every $W\in\mathfrak p_-$ there is a unique $JW\in\mathfrak p_+$ such that $(JW)_e=W_e$. This defines a linear map $J\colon\mathfrak p_-\to\mathfrak p_+$. Its definition requires surjectivity of evaluation only on $\mathfrak p_+$.

The two parallel spaces have opposite bracket signs because the torsions of $\nabla^+$ and $\nabla^-$ are opposite. More precisely,
\begin{equation}\label{eq:parallel-opposite-brackets}
 \begin{aligned}
 [U,V]&=-H^\sharp(U,V) &&(U,V\in\mathfrak p_+),\\
 [W,Z]&= H^\sharp(W,Z) &&(W,Z\in\mathfrak p_-),\\
 J[W,Z]&=-[JW,JZ] &&(W,Z\in\mathfrak p_-).
 \end{aligned}
\end{equation}
The first two identities follow from the definition of torsion. For the third, both sides belong to $\mathfrak p_+$, and their values at $e$ agree by the first two identities and the definition of $J$. Injectivity of evaluation on $\mathfrak p_+$ proves it. Thus $J$ reverses brackets, while $-J$ is a Lie algebra homomorphism.

We now define
\begin{equation}\label{eq:intrinsic-endomorphism}
 N=-T^{-1}J\pi_-|_{\lfr}\colon\lfr\longrightarrow\lfr.
\end{equation}
Each factor on the right of \eqref{eq:intrinsic-endomorphism}, with the minus sign attached to $J$, is a Lie algebra homomorphism. Hence $N$ is an endomorphism of $\lfr$. Under evaluation at $e$, the bracket of $\lfr$ is $[X,Y]_{\lfr}=-[X,Y]_{\gfr}$. Reversing both brackets in the identity $N[X,Y]_{\lfr}=[NX,NY]_{\lfr}$ shows that the same linear map is an endomorphism of $\gfr$.

The metric transported by $T$ is $R(X,Y)=g(TX,TY)$. Since $J$ preserves the values of fields at $e$, equation~\eqref{eq:intrinsic-endomorphism} gives $g(\pi_-X,\pi_-Y)=R(NX,NY)$. These pairings are constant, so this equality holds on all of $G$. Equation~\eqref{eq:projection-pairing} therefore becomes
\begin{equation}\label{eq:contraction}
 R(X,Y)-R(NX,NY)=Q(X,Y).
\end{equation}
Let $\mu$ be the minimum of $Q(X,X)$ on the $R$-unit sphere. It is positive. Since $Q\leq R$, we have $0<\mu\leq1$, and \eqref{eq:contraction} gives $R(NX,NX)\leq(1-\mu)R(X,X)$ for every $X$. Thus $N$ is strictly contractive for $R$, including the case $\mu=1$, when $N=0$. By \cref{lem:nilpotence}, some power $N^m$ vanishes. Iterating \eqref{eq:contraction} then yields
\begin{equation}\label{eq:iterated-contraction}
 R(X,Y)=\sum_{j=0}^{m-1}Q(N^jX,N^jY)+R(N^mX,N^mY)=R_N(X,Y).
\end{equation}
In particular, the inner product required by the stationary pair is determined by $N$ and $Q$ through \eqref{eq:iterated-contraction}.

\emph{Recovery of the original tensors.} We identify the underlying vector spaces of $\lfr$ and $\gfr$ by evaluation at $e$ and write $\tau X=(TX)_e$. From \eqref{eq:intrinsic-endomorphism} we have $(\pi_-X)_e=-\tau NX$. Evaluating $X=\pi_+X+\pi_-X$ therefore gives $X=\tau X-\tau NX$, or $\tau L=\id$. Since $N$ is nilpotent, $L=\id-N$ is invertible, with $L^{-1}=\id+N+\cdots+N^{m-1}$. Thus $\tau=L^{-1}$.

The metric and torsion in the parallel frame now determine their values in the original coordinates. For $X,Y,Z\in\gfr$, using the opposite bracket on $\lfr$ gives
\begin{equation}\label{eq:parametrization-evaluation}
 \begin{aligned}
 g_e(\tau X,\tau Y)&=R(X,Y),\\
 H_e(\tau X,\tau Y,\tau Z)
 &=-g([TX,TY],TZ)=-R([X,Y]_{\lfr},Z)\\
 &=R([X,Y]_{\gfr},Z).
 \end{aligned}
\end{equation}
Substituting $LX,LY,LZ$ into \eqref{eq:parametrization-evaluation}, and using $R=R_N$, gives exactly \eqref{eq:brf-pairs}. This also explains why the Cartan expression in that formula has a positive sign, although the bracket of the positive parallel fields has the negative sign in \eqref{eq:parallel-opposite-brackets}.

\emph{Construction on the simply connected cover.} Conversely, let $N$ be a Lie algebra endomorphism with $N^m=0$, and define $R_N$ and $L$ by \eqref{eq:parametrization-data}. The first term of the sum defining $R_N$ is $Q$, so $R_N$ is positive definite. For every $j$, bracket preservation gives $N^j[Z,X]=[N^jZ,N^jX]$. Ad-invariance of $Q$ then implies
\begin{equation}\label{eq:pullback-invariant-inner-products}
 Q(N^j[Z,X],N^jY)+Q(N^jX,N^j[Z,Y])=0.
\end{equation}
Summing \eqref{eq:pullback-invariant-inner-products} proves ad-invariance of $R_N$. Subtracting its defining sum with $X,Y$ replaced by $NX,NY$ gives $R_N-N^*R_N=Q$, since the final term contains $N^m=0$. Thus \eqref{eq:contraction} holds for this construction.

Let $p\colon\widetilde G\to G$ be the simply connected covering homomorphism. It is a finite cover because $G$ is compact semisimple, and $\widetilde G$ is a compact Lie group with Lie algebra $\gfr$. Lie integration gives a group homomorphism $\varphi\colon\widetilde G\to\widetilde G$ with differential $N$. Its $m$-fold iterate has zero differential, so $\varphi^m$ is the constant homomorphism to the identity. Define
\begin{equation}\label{eq:nilpotent-model-diffeomorphism}
 \begin{aligned}
 f(x)&=x\varphi(x)^{-1},\\
 f^{-1}(y)&=y\varphi(y)\cdots\varphi^{m-1}(y).
 \end{aligned}
\end{equation}
The order of the factors in \eqref{eq:nilpotent-model-diffeomorphism} is essential. Indeed, $f(x)=y$ is equivalent to $x=y\varphi(x)$. Iterating this identity gives $x=y\varphi(y)\cdots\varphi^{m-1}(y)$ because $\varphi^m(x)=e$. Conversely, if $P(y)$ denotes this ordered product, then $P(y)=y\varphi(P(y))$, and hence $f(P(y))=y$. These two identities show that $f$ is a global diffeomorphism, with the stated smooth inverse.

The homomorphism property of $\varphi$ gives
\begin{equation}\label{eq:model-translation-equivariance}
 f(ax)=a f(x)\varphi(a)^{-1},\qquad
 f(e)=e,\qquad (df)_e=\id-N=L.
\end{equation}
Let $(R_N,\rho_{R_N})$ denote the standard bi-invariant pair on $\widetilde G$. Both its metric and its Cartan form are invariant under left and right translations. Equation~\eqref{eq:model-translation-equivariance} therefore shows that $f^*(R_N,\rho_{R_N})$ is left-invariant for the original multiplication on $\widetilde G$. Evaluating its tensors at $e$ gives the two formulas in \eqref{eq:brf-pairs}. Since the standard pair is BRF and both its Bismut connections are flat, its pullback has the same properties.

Every deck transformation of $p$ is left translation by an element of the finite central kernel. Thus the left-invariant tensors just constructed descend to $G$, where they are precisely the left-invariant extensions prescribed in the theorem. Closedness, the BRF equations and flatness descend because they can be checked after pullback by a covering map. This argument concerns the tensors themselves. It does not require $\varphi$ or $f$ to induce a map on the given quotient $G$.

\emph{The cohomology class on $G$.} The pullback construction establishes the local geometric equations. To verify the required class on the original group, we exhibit an invariant primitive there. Define the two-form
\begin{equation}\label{eq:torsion-primitive}
 \beta_N(X,Y)=R_N(X,NY)-R_N(NX,Y).
\end{equation}
Its skew-symmetry follows from symmetry of $R_N$. All brackets in this part of the proof are the original bracket on $\gfr$. We write $\sum_{\rm cyc}$ for the sum over $(X,Y,Z)$, $(Y,Z,X)$ and $(Z,X,Y)$. By the invariant differential formula \eqref{eq:invariant-differential}, the definition of $\beta_N$ gives
\begin{equation}\label{eq:torsion-primitive-differential}
 \begin{aligned}
 d\beta_N(X,Y,Z)
 &=-\sum_{\rm cyc}\beta_N([X,Y],Z)\\
 &=-\sum_{\rm cyc}R_N([X,Y],NZ)
   +\sum_{\rm cyc}R_N(N[X,Y],Z)\\
 &=-\sum_{\rm cyc}R_N([X,Y],NZ)
   +\sum_{\rm cyc}R_N([NX,NY],Z).
 \end{aligned}
\end{equation}
The last equality uses the homomorphism property of $N$.

For completeness, expanding the three factors of $L=\id-N$ in $H_N$ gives
\begin{equation}\label{eq:cartan-pullback-expansion}
 \begin{aligned}
 H_N(X,Y,Z)
 ={}&R_N([X,Y],Z)\\
 &-R_N([NX,Y],Z)-R_N([X,NY],Z)-R_N([X,Y],NZ)\\
 &+R_N([NX,NY],Z)+R_N([NX,Y],NZ)+R_N([X,NY],NZ)\\
 &-R_N([NX,NY],NZ).
 \end{aligned}
\end{equation}
Ad-invariance makes $R_N([A,B],C)$ an alternating three-form, and hence it is unchanged by cyclic permutation. The three terms containing exactly one occurrence of $N$ in \eqref{eq:cartan-pullback-expansion} therefore give the first cyclic sum in \eqref{eq:torsion-primitive-differential}, with its negative sign. The three terms containing exactly two occurrences give the second cyclic sum.

Bracket preservation identifies the last term in \eqref{eq:cartan-pullback-expansion} with $(N^*R_N)([X,Y],Z)$. By \eqref{eq:contraction}, the difference between the first and last terms is the Cartan form $\rho$ evaluated on $X,Y,Z$.

Comparing \eqref{eq:torsion-primitive-differential} with \eqref{eq:cartan-pullback-expansion} proves
\begin{equation}\label{eq:torsion-expansion}
 \begin{aligned}
 d\beta_N(X,Y,Z)
 &=-\sum_{\rm cyc}R_N([X,Y],NZ)
       +\sum_{\rm cyc}R_N([NX,NY],Z)\\
 &=H_N(X,Y,Z)-\rho(X,Y,Z).
 \end{aligned}
\end{equation}
All forms in \eqref{eq:torsion-expansion} are defined by left translation on $G$ itself. Hence $H_N=\rho+d\beta_N$ proves $[H_N]=[\rho]$ on the required group, independently of whether the model diffeomorphism descends.

\emph{Uniqueness of the parameter.} We verify that the intrinsic construction \eqref{eq:intrinsic-endomorphism} recovers the endomorphism used above. Work first on $\widetilde G$, and distinguish its original pair from the standard pair to which $f$ maps it. On the standard model, write $X^R$ and $X^L$ for the right- and left-invariant fields labeled by $X\in\gfr$. The right-invariant fields are $\nabla^+$-parallel and the left-invariant fields are $\nabla^-$-parallel. This follows either directly from $\nabla^g_{X^L}Y^L=\frac12[X,Y]^L$ and $H^\sharp(X^L,Y^L)=[X,Y]^L$, or from the corresponding formulas for right-invariant fields. Each family is a global frame, so it is the full corresponding parallel space.

For $X\in\gfr$, let $\widehat X$ denote the corresponding right-invariant translation field for the original multiplication. Differentiating \eqref{eq:model-translation-equivariance} with $a=\exp(tX)$ gives
\begin{equation}\label{eq:model-projected-translation}
 \begin{aligned}
 f_*\widehat X&=X^R-(NX)^L,\\
 \pi_+(f_*\widehat X)&=X^R,\qquad
 \pi_-(f_*\widehat X)=-(NX)^L.
 \end{aligned}
\end{equation}
Here the projections on the right are those of the standard model. The involution and its projections are natural under $f$, because the defining coclosed primitives are unique and metric duality commutes with an isometry. Therefore \eqref{eq:model-projected-translation} computes the pushforwards of the original projections as well.

The point $e$ is fixed by $f$. Consequently the map that matches the values of a negative parallel field and a positive parallel field at $e$ is also transported by $f$. On the standard model this map sends $Y^L$ to $Y^R$. Thus $J$ sends the negative component $-(NX)^L$ in \eqref{eq:model-projected-translation} to $-(NX)^R$. Applying $-T^{-1}$ recovers the label $NX$. This proves that distinct endomorphisms cannot produce the same pair on the simply connected cover.

To pass to $G$, lift a translation field and its coclosed primitive through the finite covering $p$. Pullback commutes with $d$ and $\delta_g$ because the covering is a local isometry for the lifted metric. The lifted primitive is therefore the canonical primitive of the lifted field, by uniqueness on $\widetilde G$. In particular, $\Phi$, $\pi_\pm$, evaluation at $e$ and the recovered endomorphism are compatible with this cover. If two parameters produced the same pair on $G$, their pullbacks would produce the same pair on $\widetilde G$, where we have just proved that the parameters agree. This establishes bijectivity on the given group. Flatness and the critical-value formula follow from \cref{prop:critical-bound}, and replacing $H_N$ by $-H_N$ gives the negative class.
\end{proof}

The construction describes the original translations as an action by left and right translations on a bi-invariant model by $a\cdot y=ay\varphi(a)^{-1}$. The identity $f(ax)=a\cdot f(x)$ identifies this action with the usual left action on the simply connected cover. For background on transitive actions on compact group manifolds, see \cite{Oz1977}. At the Lie algebra level, the graph $\{(X,NX):X\in\gfr\}$ is complementary to the diagonal in $\gfr\oplus\gfr$, since $\id-N$ is invertible. This description by complementary subalgebras appears in \cite[Proposition~2.14]{BuDeVe2012}. The restriction imposed here by the Cartan class is the positive identity \eqref{eq:contraction}, which forces nilpotence of $N$.

\begin{remark}\label{rem:product-model}
Let $\gfr=\mathfrak k\oplus\mathfrak k$, with $\mathfrak k$ compact simple, and take $Q=Q_0\oplus Q_0$, $Q_0=-B_{\mathfrak k}$. For $N(X,Y)=(0,X)$ we have $N^2=0$, so $R_N=Q+N^*Q=2Q_0\oplus Q_0$ and $L(X,Y)=(X,Y-X)$. Substitution in \eqref{eq:brf-pairs} and \eqref{eq:torsion-primitive} gives
\begin{equation}\label{eq:product-stationary-data}
 \begin{aligned}
 g_N&=\begin{pmatrix}3&-1\\-1&1\end{pmatrix}\otimes Q_0,\\
 \beta_N((X,Y),(Z,W))&=Q_0(Y,Z)-Q_0(X,W),\qquad
 H_N=\rho+d\beta_N.
 \end{aligned}
\end{equation}
In the $Q$-orthogonal decomposition $X^{\pm}=(X,\pm X)/\sqrt2$ into diagonal and anti-diagonal subspaces, the metric matrix in \eqref{eq:product-stationary-data} becomes $\left(\begin{smallmatrix}1&1\\1&3\end{smallmatrix}\right)$, and $\beta_N(X^+,Y^-)=Q_0(X,Y)$. Thus these are the metric $g_{1,3,1}$ and two-form $\beta$ of \cite[Proposition~7.3]{FuLaSt2026}. If $d=\dim\mathfrak k$, then $\tr_{R_N}(-B_{\gfr})=d/2+d$, so its critical value is $d/4$, strictly below the standard value $d/3$.
\end{remark}

\begin{proof}[Proof of \cref{thm:rigidity}]
For a compact connected simple group, $\HH^3(G)=\mathbb R[\rho_0]$, where $Q_0=-B_{\gfr}$ and $\rho_0(X,Y,Z)=Q_0([X,Y],Z)$. See also \cite[Section~6]{LaWi2023}. Thus every possible torsion class has the form $[H]=h[\rho_0]$. We first exclude $h=0$, so that the definite-class parametrization applies to every BRF pair on this group.

If $h=0$, write $H=d\beta$. Since $H$ is harmonic, integration by parts gives $\|H\|_{L^2,\ext}^2=\langle\beta,\delta_gH\rangle_{L^2,\ext}=0$. Hence $H=0$ and the BRF equation reduces to $\Ric(g)=0$. For a Killing field $Z$, the usual integrated Bochner identity gives $\int_G|\nabla^gZ|_g^2\,dV_g=\int_G\Ric(Z,Z)\,dV_g=0$. Thus $Z$ is parallel, and its dual is a harmonic one-form. The vanishing of $\HH^1(G)$ forces $Z=0$. This contradicts the existence of the nonzero translation fields on the nontrivial group. Therefore $h\ne0$.

Set $\varepsilon=\operatorname{sign}(h)$ and $Q=|h|Q_0$. Its Cartan form is $\rho=|h|\rho_0$, and the pair $(g,\varepsilon H)$ is BRF with torsion class $[\rho]$. Apply \cref{thm:parametrization} to obtain a nilpotent endomorphism $N$ of $\gfr$. The kernel of a Lie algebra endomorphism is an ideal. Simplicity therefore implies that a nonzero endomorphism is injective, hence an automorphism in finite dimensions. Such a map cannot be nilpotent, so $N=0$. The formulas in \eqref{eq:parametrization-data} and \eqref{eq:brf-pairs} now give $g=|h|Q_0$ and $\varepsilon H=|h|\rho_0$, or $H=h\rho_0$, for the original group multiplication. Conversely, these standard pairs are BRF, with both connections $\nabla^\pm$ flat, as proved in the construction above.
\end{proof}

\begin{remark}\label{rem:simple-limit}
By \cref{thm:flow,thm:rigidity}, every invariant flow with $[H_0]=\varepsilon[\rho]$ on a compact simple group converges to $(Q,\varepsilon\rho)$. This proves the simple-group case of \cite[Conjecture~4.14]{GaSt2021} under the weaker cohomology condition.
\end{remark}
\section{Cartan initial forms and nonstandard limits}\label{sec:cartan}
On a simple group, the preceding rigidity theorem identifies the limit of every flow in the prescribed class. On a product, the question is which stationary pairs can occur as limits when $H_0=\rho$. We use the pair of Remark~\ref{rem:product-model} and the invariant family introduced in \cite[Section~7]{FuLaSt2026}. To prove Theorem~\ref{thm:cartan}, we estimate a trajectory starting away from that pair. The existence and instability of the stationary pair alone do not determine the limit of a flow starting with the Cartan form.

Let $K$ be compact and connected with nonabelian simple Lie algebra $\mathfrak k$, and put $q=-B_{\mathfrak k}$ and $Q=q\oplus q$ on $G=K\times K$. With $\beta$ as in \eqref{eq:product-beta}, write
\begin{equation}\label{eq:product-family}
 g_{x,y,z}=P\otimes q,\qquad
 P=\begin{pmatrix}x&z\\z&y\end{pmatrix}>0,\qquad
 H_r=\rho+r\,d\beta,\qquad D=xy-z^2.
\end{equation}
The condition $P>0$ means $x>0$, $y>0$ and $D>0$. These tensors are invariant under left translations and under right translations by the diagonal subgroup $\Delta K$. The adjoint representation of $K$ is absolutely irreducible. Thus an equivariant endomorphism of $\mathfrak k$ is a scalar multiple of the identity. Applying this observation to the four blocks of a bilinear form on $\mathfrak k\oplus\mathfrak k$ shows that the invariant symmetric forms are precisely $P\otimes q$. An invariant alternating form vanishes on each diagonal block and has one scalar off-diagonal block. It is therefore a multiple of $\beta$. Compare \cite[Lemma~7.1]{FuLaSt2026}.

The Ricci tensor and $H_g^2$ preserve these symmetries. The same is true of $\delta_g$, so $\delta_gH_r$ is a multiple of $\beta$. Consequently the metric equation stays within the matrices $P$, and the three-form equation stays within $\rho+\mathbb R\,d\beta$. This proves that \eqref{eq:product-family} is preserved by \eqref{eq:grf}. In what follows, we compute its four-dimensional vector field and then identify a region from which the flow converges to the nonstandard pair.

In product coordinates, choose a $q$-orthonormal basis $e_i$ of $\mathfrak k$, set $c_{ijk}=q([e_i,e_j],e_k)$, and let $e_{a,i}$ denote $e_i$ in factor $a\in\{1,2\}$. The brackets between the two factors vanish, and the normalization $q=-B_{\mathfrak k}$ gives
\begin{equation}\label{eq:product-bracket-normalization}
 [e_{a,i},e_{b,j}]=\delta_{ab}\sum_k c_{ijk}e_{a,k},\qquad
 \sum_{j,k}c_{ijk}c_{\ell jk}=\delta_{i\ell}.
\end{equation}
In particular, $\sum_{i,j,k}c_{ijk}^2=d$, where $d=\dim K$. The tensor $c$ is alternating because $q$ is ad-invariant. The components of $H_r$ are
\begin{equation}\label{eq:product-components}
 H_r(e_{a,i},e_{b,j},e_{c,k})=h_{abc}c_{ijk},\qquad
 h_{111}=h_{222}=1,\quad h_{112}=r,\quad h_{122}=-r,
\end{equation}
where $h$ is symmetric in the factor indices. Indeed, \eqref{eq:invariant-differential} gives $d\beta(e_{1,i},e_{1,j},e_{2,k})=c_{ijk}$ and $d\beta(e_{1,i},e_{2,j},e_{2,k})=-c_{ijk}$.

Although $H_r$ is alternating, the array $h$ in \eqref{eq:product-components} is symmetric. Interchanging two arguments changes the sign of $c_{ijk}$, so it must leave the factor coefficient unchanged. This accounts for the multiplicities of the mixed components in the contractions below.

Write $p^{ab}=(P^{-1})_{ab}$. The inverse metric in the fixed product basis is $g^{(a,i)(b,j)}=p^{ab}\delta_{ij}$. To apply the classical Ricci formula \cite{Mi1976}, let $E_I$ denote this basis, with a capital index representing both a factor and a Lie algebra index. Unimodularity gives
\begin{align}
 \Ric(U,V)
 &=-\frac12 B_{\mathfrak g}(U,V)
   -\frac12\sum_{I,J}g^{IJ}g([U,E_I],[V,E_J])\notag\\
 &\quad+\frac14\sum_{I,J,L,M}g^{IJ}g^{LM}
       g([E_I,E_L],U)g([E_J,E_M],V).
 \label{eq:product-milnor}
\end{align}
Compare \cite[Remark~3.2]{FuLaSt2026}. Inserting $U=e_{a,i}$ and $V=e_{b,j}$, the Killing term is $\frac12\delta_{ab}\delta_{ij}$. The second contraction is $P_{ab}p^{ab}\delta_{ij}$, with no sum over $a,b$. For the last term, brackets are nonzero only within one factor. Its contraction is therefore $\sum_{c,e}P_{ac}P_{be}(p^{ce})^2\delta_{ij}$ by \eqref{eq:product-bracket-normalization}. Thus \eqref{eq:product-milnor} and the definition of $H_g^2$ give $\Ric(g)=R\otimes q$ and $H_g^2=T\otimes q$, where
\begin{align}
 R_{ab}&=\tfrac12\delta_{ab}-\tfrac12P_{ab}p^{ab}
       +\tfrac14\sum_{c,e=1}^2P_{ac}P_{be}(p^{ce})^2,
       \label{eq:product-ricci}\\
 T_{ab}&=\sum_{c,e,u,v=1}^2h_{ace}h_{buv}p^{cu}p^{ev}.
       \label{eq:product-torsion}
\end{align}
For example, the contraction defining $T_{ab}\delta_{ij}$ contains the sum $\sum_{k,\ell}c_{ik\ell}c_{jk\ell}$. Equation~\eqref{eq:product-bracket-normalization} removes this sum and leaves exactly \eqref{eq:product-torsion}. These formulas use the full tensor contraction for $H_g^2$.

For symmetric factor arrays $h,k$, set $C_P(h,k)=\sum_{a,b,c,u,v,w=1}^2 h_{abc}k_{uvw}p^{au}p^{bv}p^{cw}$, and let $h'=\partial_rh$. The adjoint identity for $d$ uses exterior-form inner products. The factor $1/6$ for three-forms gives
\begin{equation}\label{eq:product-exterior-contractions}
 |\beta|_{\ext,g}^2=d\bigl(p^{11}p^{22}-(p^{12})^2\bigr)=\frac dD,
 \qquad
 \langle H_r,d\beta\rangle_{\ext,g}=\frac d6 C_P(h,h').
\end{equation}
Write $\delta_gH_r=\sigma\beta$. All pointwise inner products here are constant on the group. Hence the volume cancels from the integrated adjoint identity, which becomes $\sigma|\beta|_{\ext,g}^2=\langle H_r,d\beta\rangle_{\ext,g}$. By \eqref{eq:product-exterior-contractions}, $\sigma=DC_P(h,h')/6$. Since $\partial_tH_r=\dot r\,d\beta$, the reduced system is
\begin{equation}\label{eq:product-flow}
 \dot P=-2R+\tfrac12T,\qquad
 \dot r=-\frac D6C_P(h,h').
\end{equation}
The contraction formulas also give compact scalar expressions for all four equations. They will make the invariant surface and the limiting system at $x=\infty$ directly checkable. Put $A=\operatorname{adj}P$, and let $B_a=(h_{ace})_{c,e=1}^2$. Explicitly,
\begin{equation}\label{eq:product-contraction-matrices}
 A=\begin{pmatrix}y&-z\\-z&x\end{pmatrix},\qquad
 B_1=\begin{pmatrix}1&r\\r&-r\end{pmatrix},\qquad
 B_2=\begin{pmatrix}r&-r\\-r&1\end{pmatrix}.
\end{equation}
Using the matrices in \eqref{eq:product-contraction-matrices}, equation \eqref{eq:product-torsion} reads $D^2T_{ab}=\tr(B_aAB_bA)$. Define $L_1=\tr(B_1A)=y-r(x+2z)$ and $L_2=\tr(B_2A)=x+r(y+2z)$. Using $\tr(C^2)=(\tr C)^2-2\det C$ for a two-by-two matrix, and then polarizing this identity, gives
\begin{equation}\label{eq:product-torsion-factored}
 \begin{aligned}
 D^2T_{11}&=L_1^2+2r(r+1)D,\\
 D^2T_{22}&=L_2^2+2r(r-1)D,\\
 D^2T_{12}&=L_1L_2-(1+r^2)D.
 \end{aligned}
\end{equation}
For the mixed term, the coefficient of $t$ in $\det(B_1+tB_2)$ is $1+r^2$, which fixes its sign. In the same notation, \eqref{eq:product-ricci} simplifies to
\begin{equation}\label{eq:product-ricci-factored}
 \begin{aligned}
 4D^2R_{11}&=D^2+z^2(x+z)^2,\\
 4D^2R_{22}&=D^2+z^2(y+z)^2,\\
 4D^2R_{12}&=z\bigl(xy(x+y+3z)-z^3\bigr).
 \end{aligned}
\end{equation}
To record the three-form equation, set $\mathcal B=xy(x+y+4z)+2z^2(x+y+z)$. Contracting the two mixed types in $h'$ gives $D^3C_P(h,h')=3z(x-y)(x+y+z)+3r\mathcal B$. Therefore \eqref{eq:product-flow}, together with \eqref{eq:product-torsion-factored} and \eqref{eq:product-ricci-factored}, is the explicit system
\begin{equation}\label{eq:product-coordinate-flow}
 \begin{aligned}
 2D^2\dot x&=L_1^2+2r(r+1)D-D^2-z^2(x+z)^2,\\
 2D^2\dot y&=L_2^2+2r(r-1)D-D^2-z^2(y+z)^2,\\
 2D^2\dot z&=L_1L_2-(1+r^2)D
                  -z\bigl(xy(x+y+3z)-z^3\bigr),\\
 2D^2\dot r&=-z(x-y)(x+y+z)-r\mathcal B.
 \end{aligned}
\end{equation}
All coefficients are rational functions, smooth on $P>0$. In particular, the last equation of \eqref{eq:product-coordinate-flow} gives
\begin{equation}\label{eq:cartan-derivative}
 \left.\dot r\right|_{r=0}
 =-\frac{z(x-y)(x+y+z)}{2D^2}.
\end{equation}
The right-hand side of \eqref{eq:cartan-derivative} is generally nonzero. Thus the initial equality $H_0=\rho$ need not persist along the flow.

The two stationary points of interest are $(1,1,0,0)$ and $p=(3,1,-1,1)$. The first represents $(Q,\rho)$, and the second represents the pair $(g_*,H_*)$ in Theorem~\ref{thm:cartan}. Both lie on the surface
\begin{equation}\label{eq:product-surface}
 \Sigma=\{y=1,\ z=-r,\ x>r^2\}.
\end{equation}
To verify the invariance of \eqref{eq:product-surface}, substitute $y=1$ and $z=-r$ in \eqref{eq:product-coordinate-flow}. The numerator of $\dot y$ vanishes, and the numerators of $\dot z$ and $\dot r$ sum to zero. Hence $\dot y=0$ and $(z+r)^{\boldsymbol\cdot}=0$ there. Since $D=x-r^2$, the remaining equations are
\begin{align}
 \dot x&=-\frac{(x+1)(x-1+2r^3-4r^2)}{2(x-r^2)^2},
 \label{eq:surface-x}\\
 \dot r&=-\frac{r(r-1)((2r-1)x-2r^2-1)}{2(x-r^2)^2}.
 \label{eq:surface-r}
\end{align}
We must also control solutions that are close to $\Sigma$ but do not lie on it. Define $\nu=y-1$, $w=z+r$ and $\eta=(\nu^2+w^2)^{1/2}$. Holding $x,r$ fixed, let $J(x,r)$ be the derivative at $(\nu,w)=(0,0)$ of the right-hand side of $(\dot\nu,\dot w)$. Differentiating \eqref{eq:product-coordinate-flow} gives
\begin{equation}\label{eq:product-normal-matrix}
 J(x,r)=-\id+\frac{1}{2(x-r^2)^2}
 \begin{pmatrix}
 2r^3(r-1)&2r(r-1)(2r-1)\\
 -r^2(x+1)&-(2r^2-2r+1)x+2r^2(r-1)^2-2r+1
 \end{pmatrix}.
\end{equation}
In this differentiation, $\partial_\nu=\partial_y$ and $\partial_w=\partial_z$, because $r$ is fixed. The normal numerator vanishes on $\Sigma$, so differentiating the common denominator contributes no additional term. At $p$, \eqref{eq:product-normal-matrix} becomes
\begin{equation}\label{eq:normal-p}
 J_p=\begin{pmatrix}-1&0\\-1/2&-3/2\end{pmatrix}.
\end{equation}

The system also has a regular limit as $x\to\infty$. This limit describes the first part of a trajectory whose initial $x$ coordinate is large. Set $s=1/x$. Each numerator in \eqref{eq:product-coordinate-flow} has degree at most two in $x$. Dividing its numerator and denominator by $x^2$ writes each right-hand side as a polynomial in $(s,y,z,r)$ divided by $2(y-sz^2)^2$. In particular, the expression for $\dot x$ itself extends smoothly to $s=0$, with value $(r^2-y^2-z^2)/(2y^2)$. The actual first coordinate of the new vector field is $\dot s=-s^2\dot x$. It vanishes at $s=0$, and the other boundary equations are
\begin{equation}\label{eq:product-boundary}
 \dot y=\frac{1-y^2}{2y^2},\qquad
 \dot z=-\frac{yz+r}{2y^2},\qquad
 \dot r=-\frac{yr+z}{2y^2}.
\end{equation}
Thus the vector field in $(s,y,z,r)$ extends smoothly to an open neighborhood of every point with $s=0$ and $y>0$. At that boundary, \eqref{eq:product-boundary} gives $\dot w=-(y+1)w/(2y^2)$. Linearizing its $y,w$ equations at $y=1,w=0$ gives $-\id$. Equivalently, the explicit expression \eqref{eq:product-normal-matrix} gives
\begin{equation}\label{eq:normal-infinity}
 J(x,r)=-\id+O(x^{-1})
\end{equation}
uniformly for $r$ in any fixed compact interval. Indeed, the numerator of the correction in \eqref{eq:product-normal-matrix} is at most linear in $x$, whereas the denominator is quadratic. This uniformity will allow a single transverse estimate on the whole region $x\ge X$.

The linearization at $p$ has a zero eigenvalue, so it does not determine stability. The following lemma gives a region of attraction that allows small transverse errors.

\begin{lemma}\label{lem:cartan-basin}
For every fixed $X>3$ there are $\varepsilon_X,c_X>0$ such that a solution with initial values $x=X$, $0<r-1<\varepsilon_X$ and $\eta<c_X(r-1)$ exists for all positive time and converges to $p$.
\end{lemma}
\begin{proof}
\emph{Local coordinates near the stationary point.}
Put $u=x-2r-1$ and $v=r-1$. The purpose of $u$ is to separate a strictly contracting direction from the direction with zero linear term. On $\Sigma$, subtracting twice \eqref{eq:surface-r} from \eqref{eq:surface-x} gives
\begin{equation}\label{eq:product-uv}
 \dot u=-\lambda u,\qquad
 \dot v=-\alpha v^2-\gamma uv,
\end{equation}
where
\begin{equation}\label{eq:product-uv-coefficients}
 \lambda=\frac{-2r^3+2r^2+x+1}{2(x-r^2)^2},\qquad
 \alpha=\frac{r(r+1)}{(x-r^2)^2},\qquad
 \gamma=\frac{r(2r-1)}{2(x-r^2)^2}.
\end{equation}
For the second identity, the factor in the numerator of \eqref{eq:surface-r} becomes $(2r-1)x-2r^2-1=(2r-1)u+2(r+1)v$. At $p$, the coefficients in \eqref{eq:product-uv-coefficients} are $1/2,1/2,1/8$. Thus $u$ has a negative linear term, while the leading term of $\dot v$ along $u=0$ is negative and quadratic when $v>0$.

Off $\Sigma$, use the same coefficient functions of $(x,r)$ and write the right-hand sides of \eqref{eq:product-uv} with additional errors $E_u,E_v$. Smoothness and invariance of $\Sigma$ give $|E_u|+|E_v|\le L\eta$ on a fixed small neighborhood of $p$. This is a uniform mean-value estimate in $(\nu,w)$, because the two errors vanish when $\nu=w=0$.

To estimate the normal variables, write $\mathbf n=(\nu,w)^{\mathsf T}$. Their equation has the form $\dot{\mathbf n}=J(x,r)\mathbf n+O(|\mathbf n|^2)$. The largest eigenvalue of the symmetric part of $J_p$ in \eqref{eq:normal-p} is $(-5+\sqrt2)/4<-3/4$. Hence, after shrinking the neighborhood,
\begin{equation}\label{eq:product-local}
 \dot\eta\le-\tfrac34\eta,\qquad
 \tfrac25<\lambda<\tfrac35,\qquad \alpha,\gamma>0.
\end{equation}
Indeed, for $\eta>0$ one computes $\dot\eta=\langle\mathbf n,\dot{\mathbf n}\rangle/\eta$ and absorbs the continuous change in $J$ and the quadratic remainder. At $\eta=0$, invariance of $\Sigma$ applies. The norm inequalities can also be interpreted using upper right derivatives.

\emph{A preserved cone and a neighborhood that cannot be left.}
Choose $C_0>10L+1$. Shrink the neighborhood again so that $\alpha v+\gamma u<1/4$ whenever $u,v\ge0$ there. We claim that the cone
\begin{equation}\label{eq:product-cone}
 u\ge C_0\eta,\qquad v\ge C_0\eta
\end{equation}
is preserved while the solution remains in this neighborhood. On the first face of \eqref{eq:product-cone}, where $u=C_0\eta$ and $\eta>0$, equations \eqref{eq:product-uv} and \eqref{eq:product-local} imply
\begin{equation}\label{eq:product-cone-faces}
 \begin{aligned}
 (u-C_0\eta)^{\boldsymbol\cdot}
 &\ge \bigl((\tfrac34-\tfrac35)C_0-L\bigr)\eta>0,\\
 (v-C_0\eta)^{\boldsymbol\cdot}
 &\ge \bigl(\tfrac34C_0-L-C_0(\alpha v+\gamma u)\bigr)\eta>0
 \quad\text{when }v=C_0\eta.
 \end{aligned}
\end{equation}
These inequalities point into the cone. If $\eta=0$, the equations on $\Sigma$ show that both $u=0$ and $v=0$ are invariant. Thus uniqueness prevents crossing those parts of the boundary as well.

Choose $\delta>0$ so small that the closed box defined by $0\le u,v\le\delta$ and $\eta\le\delta$ is contained in this neighborhood and in $P>0$. Until a first exit from this box, a solution in the cone satisfies
\begin{equation}\label{eq:product-box-bounds}
 \begin{aligned}
 \eta(t)&\le\eta(0)e^{-3t/4},\\
 u(t)&\le u(0)+L\int_0^t\eta(\sigma)\,d\sigma
       \le u(0)+\tfrac43L\eta(0),\\
 v(t)&\le v(0)+L\int_0^t\eta(\sigma)\,d\sigma
       \le v(0)+\tfrac43L\eta(0).
 \end{aligned}
\end{equation}
The negative terms in \eqref{eq:product-uv} were discarded here because $u,v,\alpha,\gamma$ are nonnegative. Take initial values with $u(0),v(0)\le\delta/2$, $\eta(0)<\delta/2$ and $\frac43L\eta(0)<\delta/4$. Then \eqref{eq:product-box-bounds} keeps both $u$ and $v$ below $3\delta/4$ and keeps $\eta$ below $\delta/2$. Equation~\eqref{eq:product-cone-faces} excludes exit across the lower faces. A first exit is therefore impossible. Since the box is a compact subset of $P>0$, the ODE continuation theorem gives existence for all positive time.

It remains to identify the limit. The inequality $\dot u\le-2u/5+L\eta$ and variation of constants give
\begin{equation}\label{eq:product-local-convergence}
 \begin{aligned}
 u(t)&\le u(0)e^{-2t/5}
       +L\eta(0)\int_0^t e^{-2(t-\sigma)/5}e^{-3\sigma/4}\,d\sigma,\\
 \alpha_0\int_0^t v(\sigma)^2\,d\sigma
 &\le v(0)+\tfrac43L\eta(0),
 \end{aligned}
\end{equation}
where $\alpha_0>0$ is a lower bound for $\alpha$ on the box. The first line tends to zero. The second follows by integrating $\dot v\le-\alpha_0v^2+L\eta$ and using $v(t)\ge0$. The vector field and $v$ are bounded on the box, so $v^2$ is uniformly continuous. Its integrability in \eqref{eq:product-local-convergence} implies $v(t)\to0$. Otherwise uniform continuity would give infinitely many disjoint intervals on which $v^2$ has a fixed positive lower bound. Together with $\eta(t)\to0$, this proves convergence to $p$ for the local initial values just described.

\emph{Entry from the section $x=X$.}
Consider first the solution on $\Sigma$ with $r=1$ and $x(0)=X$. Its equation is
\begin{equation}\label{eq:product-reference-trajectory}
 \dot x=-\frac{(x+1)(x-3)}{2(x-1)^2}.
\end{equation}
For $x>3$, this derivative is negative. The solution stays between $3$ and $X$, so it exists for all positive time. Its limit must be $3$, since the right-hand side of \eqref{eq:product-reference-trajectory} is strictly negative at every point of $(3,X]$. It does not reach $3$ in finite time by uniqueness.

Choose a finite time $\tau$ such that the reference value $u_*=x(\tau)-3$ lies in $(0,\delta/4)$. Along the reference trajectory $v=0$, the linearized $v$ equation from \eqref{eq:product-uv} is $\dot v=-\gamma u v$. Its solution multiplier at time $\tau$ is
\begin{equation}\label{eq:product-entry-multiplier}
 m=\exp\left(-\int_0^\tau\gamma u\,dt\right)>0.
\end{equation}
A change in the initial $u$ does not add a first-order forcing term to this equation, because the derivative of $-\alpha v^2-\gamma uv$ with respect to $u$ vanishes at $v=0$.

The reference segment is compact and lies in $P>0$. Smooth dependence gives a common time-$\tau$ map on a neighborhood of its initial point. By \eqref{eq:product-entry-multiplier} and invariance of $\Sigma$, there are constants $C_1,C_2,C_3$ such that nearby initial values with $x=X$ satisfy
\begin{equation}\label{eq:product-finite-entry}
 \begin{aligned}
 |v(\tau)-m v(0)|&\le C_1\bigl(v(0)^2+\eta(0)\bigr),\\
 \eta(\tau)&\le C_2\eta(0),\\
 |u(\tau)-u_*|&\le C_3\bigl(|v(0)|+\eta(0)\bigr).
 \end{aligned}
\end{equation}
For the first estimate, restrict first to $\Sigma$ and use the Taylor expansion in $v(0)$. The change caused by nonzero $(\nu(0),w(0))$ is then bounded by the mean value theorem. The second estimate follows because the normal component of the time map vanishes identically on $\Sigma$.

Choose $c_X>0$ so that $C_1c_X<m/4$ and $C_0C_2c_X<m/4$. Next choose $\varepsilon_X>0$ so small that $C_1\varepsilon_X<m/4$ and the last line of \eqref{eq:product-finite-entry} keeps $u(\tau)$ in $(u_*/2,3u_*/2)$. If $0<v(0)<\varepsilon_X$ and $\eta(0)<c_Xv(0)$, the first two lines imply $v(\tau)\ge m v(0)/2>C_0\eta(\tau)$. Decreasing $\varepsilon_X$ further ensures $u(\tau)>C_0\eta(\tau)$ and all the smaller-box bounds used in \eqref{eq:product-box-bounds}. Thus the perturbed trajectory enters the preserved cone in that box at time $\tau$. Smooth dependence also guarantees $P>0$ throughout $[0,\tau]$. The local argument completes the proof.
\end{proof}

The initial data of Theorem~\ref{thm:cartan} have $r=0$, whereas Lemma~\ref{lem:cartan-basin} requires $r>1$. For large initial values of $x$, the boundary system \eqref{eq:product-boundary} first moves the trajectory close to $\Sigma$ with $r$ near $3/2$. We then compare the decay of $r-1$ with the faster transverse decay while $x$ decreases to a fixed value.

\begin{lemma}\label{lem:cartan-strip}
There are $X\ge16$, $0<\eta_*\le1/4$ and $L>0$ such that on $x\ge X$, $1\le r\le2$ and $\eta\le\eta_*$ the metric is positive definite and
\begin{align}
 -1\le\dot x&\le-\tfrac14, &
 \dot\eta&\le-\tfrac12\eta,\label{eq:product-strip}\\
 \dot r&=-\frac{a(x,r)}x(r-1)+e, &
 \tfrac14\le a(x,r)&\le4,\qquad |e|\le L\eta.
 \label{eq:product-strip-r}
\end{align}
\end{lemma}
\begin{proof}
On $\Sigma$, equation \eqref{eq:surface-x} tends uniformly to $-1/2$ as $x\to\infty$ for $1\le r\le2$. Equation~\eqref{eq:surface-r} has the required form with
\begin{equation}\label{eq:product-strip-coefficient}
 \begin{aligned}
 a(x,r)&=\frac{x r((2r-1)x-2r^2-1)}{2(x-r^2)^2},\\
 a(1/s,r)&=\frac{r((2r-1)-s(2r^2+1))}{2(1-sr^2)^2}
          \longrightarrow\frac{r(2r-1)}2.
 \end{aligned}
\end{equation}
The limit in \eqref{eq:product-strip-coefficient} lies in $[1/2,3]$. Hence $1/4\le a\le4$ once $X$ is sufficiently large. At the same time, \eqref{eq:normal-infinity} allows us to arrange that the largest eigenvalue of the symmetric part of $J(x,r)$ is at most $-3/4$ for all $x\ge X$, $1\le r\le2$.

It remains to make these estimates uniform off $\Sigma$. Use coordinates $(s,r,\nu,w)$ and set $\mathbf n=(\nu,w)^{\mathsf T}$. The compact set $s=0$, $1\le r\le2$, $\mathbf n=0$ lies in the domain $y-sz^2>0$ of the extended vector field. Choose a fixed compact neighborhood within this domain. The first two derivatives of the vector field in $\mathbf n$ are bounded there. The expression for $\dot x$, although not a component of the vector field in the $s$ coordinate, is also smooth on this neighborhood by the discussion preceding \eqref{eq:product-boundary}.

Denote the right-hand sides of the $x$ equation, the $r$ equation and the normal equations by $F_x$, $F_r$ and $F_\perp$, respectively. Taylor's formula and $F_\perp(s,r,0)=0$ give constants $C,L$, independent of $x\ge X$, for which
\begin{equation}\label{eq:product-uniform-remainders}
 \begin{aligned}
 |F_x(s,r,\mathbf n)-F_x(s,r,0)|&\le C\eta,\\
 |F_r(s,r,\mathbf n)-F_r(s,r,0)|&\le L\eta,\\
 |F_\perp(s,r,\mathbf n)-J(1/s,r)\mathbf n|&\le C\eta^2.
 \end{aligned}
\end{equation}
At $s=0$ the last line is understood using the continuous extension of $J$. Enlarge $X$ so that $-5/8\le F_x(s,r,0)\le-3/8$ whenever $0\le s\le1/X$. Then choose $\eta_*>0$ small enough that $C\eta_*\le1/8$ and the relevant region stays in the chosen compact neighborhood. The first line of \eqref{eq:product-uniform-remainders} gives $-1\le\dot x\le-1/4$. Taking the inner product of the last line with $\mathbf n$ gives $\dot\eta\le(-3/4+C\eta)\eta\le-\eta/2$. The middle line gives the error $e$ in \eqref{eq:product-strip-r}. This proves both estimates with constants chosen before any initial value of $x$ is specified.

Finally, decrease $\eta_*$ to at most $1/4$ and increase $X$ to at least $16$. Then $y=1+\nu\ge3/4$ and $|z|=|w-r|\le9/4$. Therefore
\begin{equation}\label{eq:product-strip-positivity}
 D=xy-z^2\ge\frac34x-\frac{81}{16}>0.
\end{equation}
This also gives a uniform positive lower bound for $D$ whenever $x\ge X$, and completes the proof.
\end{proof}

\begin{proof}[Proof of Theorem~\ref{thm:cartan}]
\emph{A finite initial interval near $s=0$.}
Fix the constants in Lemma~\ref{lem:cartan-strip}. Choose $0<\eta_0<\eta_*$ such that $4L\eta_0<1/6$. The prescribed initial data are $x(0)=M$, $y(0)=1$, $z(0)=-3$ and $r(0)=0$, with $M>9$. In the compactified coordinate, only $s(0)=1/M$ depends on $M$.

At $s=0$, the solution of \eqref{eq:product-boundary} with initial values $(y,z,r)=(1,-3,0)$ is
\begin{equation}\label{eq:product-boundary-solution}
 y=1,\qquad z=-\tfrac32(1+e^{-t}),\qquad
 r=\tfrac32(1-e^{-t}).
\end{equation}
Indeed, the first boundary equation fixes $y=1$. The remaining equations then give $(z+r)^{\boldsymbol\cdot}=-(z+r)$ and $(z-r)^{\boldsymbol\cdot}=0$, with initial values $z+r=z-r=-3$. Solving these two scalar equations gives \eqref{eq:product-boundary-solution}. Along this boundary trajectory, $r\to3/2$ and $\eta=|z+r|=3e^{-t}\to0$.

Choose a fixed time $T>\max\{\log9,\log(6/\eta_0)\}$. For the boundary solution, this ensures $1/3<r(T)-1<2/3$ and $\eta(T)<\eta_0/2$. The entire segment $0\le t\le T$ lies in the smooth domain $y-sz^2>0$. Continuous dependence on the initial value $s(0)=1/M$ therefore gives, for every sufficiently large $M$,
\begin{equation}\label{eq:product-entry}
 x(T)=\ell_M\longrightarrow\infty,\qquad
 \tfrac13<r(T)-1<\tfrac23,\qquad \eta(T)<\eta_0.
\end{equation}
The time $T$ is fixed before $M$ is taken large. Since $s=0$ is invariant, uniqueness prevents a nearby solution with $s(0)>0$ from reaching $s=0$ in finite time. Smooth dependence keeps $y-sz^2>0$ throughout this fixed interval. Thus the nearby trajectory corresponds to a positive definite metric on $[0,T]$. The extended expression for $\dot x$ is uniformly bounded on a common compact neighborhood of these segments. Integrating it shows more precisely that $\ell_M=M+O(1)$.

\emph{Passage through the region $x\ge X$.}
Restart time at $T$ with the initial bounds in \eqref{eq:product-entry}, so $x(0)=\ell_M$ for the following estimates. Set $v=r-1$ and
\begin{equation}\label{eq:product-integrating-factor}
 \mathcal A(t)=\int_0^t\frac{a(x(\sigma),r(\sigma))}{x(\sigma)}\,d\sigma.
\end{equation}
Until the trajectory leaves the region of Lemma~\ref{lem:cartan-strip}, the equation $\dot v=-\mathcal A'v+e$ gives
\begin{equation}\label{eq:product-variation}
 v(t)=e^{-\mathcal A(t)}
       \left(v(0)+\int_0^t e^{\mathcal A(\sigma)}e(\sigma)\,d\sigma\right).
\end{equation}
By \eqref{eq:product-strip} and \eqref{eq:product-strip-r}, we have $\mathcal A'\le4/X\le1/4$ and $|e(t)|\le L\eta_0e^{-t/2}$. Since $\mathcal A(0)=0$, this gives
\begin{equation}\label{eq:product-error-integral}
 \left|\int_0^t e^{\mathcal A(\sigma)}e(\sigma)\,d\sigma\right|
 \le L\eta_0\int_0^t e^{-\sigma/4}\,d\sigma
 \le4L\eta_0<\frac16.
\end{equation}
By \eqref{eq:product-error-integral}, the term in parentheses in \eqref{eq:product-variation} stays between $1/6$ and $5/6$. In particular, $0<v(t)<5/6$, because $\mathcal A(t)\ge0$. This excludes exit through $r=1$ or $r=2$. Also $\eta(t)\le\eta(0)e^{-t/2}<\eta_*$ excludes exit through the normal boundary.

For each fixed $M$, the part of the region with $X\le x\le\ell_M$ is compact. Equation~\eqref{eq:product-strip-positivity} keeps this compact set strictly inside $P>0$. The ODE can therefore be continued as long as $x\ge X$ and the preceding estimates hold. The bound $\dot x\le-1/4$ forces $x$ to reach $X$ at a finite time $t_M$. The bounds on $\dot x$ give
\begin{equation}\label{eq:product-passage-time}
 \ell_M-X\le t_M\le4(\ell_M-X),\qquad
 \log\frac{\ell_M}{X}
 \le\int_0^{t_M}\frac{dt}{x(t)}
 \le4\log\frac{\ell_M}{X}.
\end{equation}
For the integral estimate, $x$ is strictly decreasing and $1\le1/(-\dot x)\le4$. Changing variables from $t$ to $x$ proves the second pair of inequalities in \eqref{eq:product-passage-time}.

\emph{Entry into the attracting region at $x=X$.}
Using $1/4\le a\le4$ in \eqref{eq:product-integrating-factor} and then \eqref{eq:product-passage-time}, we obtain
\begin{equation}\label{eq:product-integrating-bounds}
 \frac14\log\frac{\ell_M}{X}\le\mathcal A(t_M)
 \le16\log\frac{\ell_M}{X}.
\end{equation}
The bounds on the parenthesis in \eqref{eq:product-variation}, together with \eqref{eq:product-integrating-bounds}, yield
\begin{align}
 \frac16\left(\frac X{\ell_M}\right)^{16}
 &\le r(t_M)-1
 \le\frac56\left(\frac X{\ell_M}\right)^{1/4},
 \label{eq:product-polynomial}\\
 \eta(t_M)&\le\eta_0 e^{-(\ell_M-X)/2}.
 \label{eq:product-exponential}
\end{align}
Thus the displacement $r-1$ remains positive and has a polynomial lower bound in $\ell_M$. In contrast, the normal error has the exponential upper bound \eqref{eq:product-exponential}. More explicitly, \eqref{eq:product-polynomial} implies
\begin{equation}\label{eq:product-relative-error}
 \frac{\eta(t_M)}{r(t_M)-1}
 \le6\eta_0\left(\frac{\ell_M}{X}\right)^{16}
          e^{-(\ell_M-X)/2}\longrightarrow0.
\end{equation}
The upper bound in \eqref{eq:product-polynomial} also gives $r(t_M)-1\to0$. For the fixed value of $X$ already chosen, equations \eqref{eq:product-polynomial} and \eqref{eq:product-relative-error} eventually satisfy both initial inequalities in Lemma~\ref{lem:cartan-basin}. Every sufficiently large $M$ therefore gives a solution existing for all positive time and converging to $p$.

The four coefficients in \eqref{eq:product-family} then converge to $(3,1,-1,1)$. Since the corresponding invariant tensors lie in fixed finite-dimensional spaces, this is smooth convergence on the compact group and gives precisely \eqref{eq:cartan-limit}. All choices above were made before the final lower threshold for $M$. The initial determinant is $M-9$, so increasing that threshold if necessary gives $M_0>9$.
\end{proof}

The same reduced equations identify an exact trajectory joining the two stationary pairs. It is the curve used to prove instability in \cite[Theorem~7.7]{FuLaSt2026}. Its invariance under the full flow gives the eternal solution suggested in \cite[Remark~7.8]{FuLaSt2026}.

\begin{proposition}\label{prop:eternal}
The curve $(x,y,z,r)=(1+2r,1,-r,r)$ is invariant under \eqref{eq:product-flow}, with induced equation
\begin{equation}\label{eq:eternal-speed}
 \dot r=-\frac{r(r-1)^2(r+1)}{(1+2r-r^2)^2}.
\end{equation}
For $0<r<1$ it gives a nonconstant solution on all $t\in\mathbb R$, whose backward limit is $(g_*,H_*)$ and whose forward limit is $(Q,\rho)$. For $1<r<1+\sqrt2$ its forward limit is $(g_*,H_*)$.
\end{proposition}
\begin{proof}
On $\Sigma$, equation \eqref{eq:product-uv} preserves $u=x-2r-1=0$. Substituting $x=1+2r$ in \eqref{eq:surface-r} gives \eqref{eq:eternal-speed}, and then $\dot x=2\dot r$. This verifies invariance under all four equations, including the metric equation.

The determinant on the curve is $D(r)=1+2r-r^2$. It is positive for $1-\sqrt2<r<1+\sqrt2$, in particular on both stated intervals. Write \eqref{eq:eternal-speed} as $\dot r=-q_0(r)$, so that $q_0$ is positive on each interval. For $r_0\in(0,1)$, separation of variables gives
\begin{equation}\label{eq:eternal-time-integral}
 t-t_0=\int_{r(t)}^{r_0}
       \frac{(1+2\xi-\xi^2)^2}{\xi(\xi-1)^2(\xi+1)}\,d\xi.
\end{equation}
There is no zero of $q_0$ in $(0,1)$, so $r$ is strictly decreasing. On each compact subinterval the scalar vector field is smooth, and continuation can fail only by approaching an endpoint. The endpoint behavior is determined by
\begin{equation}\label{eq:eternal-endpoint-speeds}
 \lim_{r\to0^+}\frac{q_0(r)}r=1,\qquad
 \lim_{r\to1}\frac{q_0(r)}{(r-1)^2}=\frac12.
\end{equation}
By \eqref{eq:eternal-endpoint-speeds}, the integrand in \eqref{eq:eternal-time-integral} is comparable to $1/\xi$ near $0$ and to $2/(\xi-1)^2$ near $1$. Both endpoint integrals diverge. Hence $r(t)\to0$ as $t\to+\infty$, while $r(t)\to1$ as $t\to-\infty$. The endpoint coefficient vectors are $(1,1,0,0)$ and $(3,1,-1,1)$, which gives the asserted forward and backward limits and an eternal nonconstant solution.

If $1<r_0<1+\sqrt2$, the same scalar equation has negative speed and keeps $r$ in $(1,r_0]$ in forward time. Its only possible forward limit is $1$, since the speed is strictly negative away from that endpoint. The second limit in \eqref{eq:eternal-endpoint-speeds} shows that the time needed to reach $1$ is infinite. Thus the forward solution exists for all positive time and converges to $(g_*,H_*)$.
\end{proof}

For comparison with \cite[Section~7]{FuLaSt2026}, the metric coordinates in the decomposition into diagonal and anti-diagonal subspaces are $a=\tfrac12(x+y+2z)$, $b=\tfrac12(x+y-2z)$ and $\ell=\tfrac12(x-y).$ The curve in Proposition~\ref{prop:eternal} becomes $(a,b,\ell,r)=(1,1+2r,r,r)$. Setting $r=1-s$ gives their curve $(1,3-2s,1-s,1-s)$. Its nonconstant points have $H_r\ne\rho$. The trajectory estimates above are needed to obtain the Cartan initial condition in Theorem~\ref{thm:cartan}.
\section*{Acknowledgements}
This work was supported by the National Natural Science Foundation of China (12571025 and 12131012) and the Natural Science Research of Jiangsu Education Institutions of China (No. 23KJB110016).

\section*{Declaration of generative AI and AI-assisted technologies in the manuscript preparation process}
During the preparation of this manuscript, the authors used OpenAI's Codex to assist with literature searches, drafting and checking mathematical arguments and calculations, and improving the organization and language of the text. The authors take full responsibility for the content of this article.

\end{document}